\documentclass[10pt]{article}

\usepackage{fancyhdr}
\usepackage{amsmath}
\usepackage{amssymb}
\usepackage{array}
\usepackage{amsthm}
\usepackage{longtable}
\usepackage{authblk}
\usepackage{graphicx}
\usepackage{subfigure}
\usepackage{enumerate}
\usepackage{hyperref}

\usepackage{epsfig}

\usepackage[algoruled,linesnumbered]{algorithm2e}

\newcommand{\beq}{\begin{equation}}
\newcommand{\eeq}{\end{equation}}
\newcommand{\beqar}{\begin{eqnarray}}
\newcommand{\eeqar}{\end{eqnarray}}
\newcommand{\bal}{\begin{aligned}}
\newcommand{\eal}{\end{aligned}}

\newcommand{\bfYset}{\boldsymbol{\cal Y}}
\newcommand{\bfy}{\boldsymbol{y}}

\newcommand{\bflambda}{\boldsymbol{\lambda}}
\newcommand{\bfw}{\boldsymbol{w}}
\newcommand{\bfv}{\boldsymbol{v}}
\newcommand{\bfQ}{\boldsymbol{Q}}
\newcommand{\bfq}{\boldsymbol{q}}
\newcommand{\bfA}{\boldsymbol{A}}
\newcommand{\bfE}{\boldsymbol{E}}
\newcommand{\bfb}{\boldsymbol{b}}
\newcommand{\bfM}{\boldsymbol{M}}
\newcommand{\bfN}{\boldsymbol{N}}
\newcommand{\bfR}{\boldsymbol{R}}
\newcommand{\bfZ}{\boldsymbol{Z}}
\newcommand{\bfI}{\boldsymbol{I}}
\newcommand{\R}{\mathsf{R}}
\newcommand{\symmats}{\mathsf{S}}
\newcommand{\refl}{\mathbb{R}}
\newcommand{\proj}{\mathbb{P}}
\newcommand{\half}{\frac{1}{2}}

\newcommand{\bfu}{\boldsymbol{u}}
\newcommand{\bfe}{\boldsymbol{e}}
\newcommand{\bfxi}{\boldsymbol{\xi}}

\newcommand{\bfmu}{\boldsymbol{\mu}}

\newcommand{\bfnu}{\boldsymbol{\nu}}

\newcommand{\activeset}{{\cal A}}
\newcommand{\stacktwo}[2]{(#1,#2)}
\newcommand{\stackthr}[3]{(#1,#2,#3)}

\newtheorem{theorem}{Theorem}

\newtheorem{lemma}{Lemma}

\newtheorem{assumption}{Assumption}

\title{ADMM for Convex Quadratic Programs: Q-Linear Convergence and Infeasibility Detection} 

\author{Arvind U. Raghunathan \; Stefano Di Cairano \\
Mitsubishi Electric Research Laboratories\\
201 Broadway, Cambridge, MA 02139 \\
Email:\texttt{raghunathan@merl.com},\texttt{dicairano@merl.com}}

\begin{document}

\maketitle




\begin{abstract}
In this paper,  we analyze the convergence of Alternating Direction Method of Multipliers (ADMM) 
on convex quadratic programs (QPs) with linear equality and bound constraints.  The ADMM formulation alternates between an equality 
constrained QP and a projection on the bounds. Under the assumptions of 
(i) positive definiteness of the Hessian of the objective projected on the null space of 
equality constraints (reduced Hessian), and (ii) linear independence constraint qualification holding at the optimal solution, we derive an upper bound on 
the rate of convergence to the solution at each iteration. In particular, we provide an explicit characterization of the rate of convergence in terms of: (a) the eigenvalues of the 
reduced Hessian, (b) the cosine of the Friedrichs angle between the subspace spanned by equality constraints and the subspace spanned by the gradients of the components 
that are active at the solution and (c) the distance of the inactive components of solution from the bounds. Using this analysis we show that if the QP is feasible, the iterates 
converge at a Q-linear rate and prescribe an optimal setting for the ADMM step-size parameter. 
For infeasible QPs, we show that the primal variables in ADMM converge to minimizers of the Euclidean distance 
between the hyperplane defined by the equality constraints and the convex set defined by the bounds.  The multipliers for the bound constraints are shown to diverge along the 
range space of the equality constraints.  Using this characterization, we also propose a termination criterion for ADMM.  Numerical examples are provided to illustrate the theory 
through experiments.
\end{abstract}


\section{Introduction}

We consider the solution of convex quadratic program (QP),
\beq\bal
\min &\; \bfq^T\bfy +\half\bfy^T\bfQ\bfy \\
\mbox{s.t.} &\; \bfA\bfy = \bfb \\
&\; \bfy \in \bfYset 
\eal\label{qpform}
\eeq
where, $\bfy \in \R^n$, $\bfQ \succeq 0$ is symmetric positive semidefinite, $\bfYset = [\underline{\bfy},\overline{\bfy}]$ are box constraints with $-\infty \leq \underline{\bfy}_i < \overline{\bfy}_i \leq \infty$ and $\bfA \in \R^{m\times n}$ is full row rank.  In particular, we consider the case where 
$\bfQ$ is positive definite on the null space of the equality constraints.  A number of problems arising from computer vision, compressive sensing, control, finance, machine learning, seismology among others can be cast as~\eqref{qpform}.  Further, solution of QP serves as the direction determining step in nonlinear optimization algorithms such as sequential quadratic programming~\cite{NocedalWrightBook} (see~\cite{GilWon12} for a recent survey).  Efficient solution of QPs has been the subject of several papers spanning the past few decades.    
There is an abundance of literature on algorithms for solution of QPs (refer to a recent survey  by Gill and Wong~\cite{GilWon13}).  Interior point and active-set methods are two 
alternative approaches to the handling of inequality constraints in QPs. A number of active-set algorithms have been proposed for 
strictly convex QPs, see, e.g., Fletcher~\cite{FlegenQP}, Goldfarb and Idnani~\cite{GolIdn83}, Gill et al.~\cite{GilGouMur84}, Powell~\cite{Pow85}, 
Gould~\cite{GougenQP}, Mor\'{e} and Toraldo~\cite{MorTor91}, Gill, Murray and Saunders~\cite{GilMurSau97} and Bartlett and Biegler~\cite{BarBie06}. 
Wright~\cite{WrightIPM} describes a number of interior point 
algorithms for convex QPs. A class of algorithms that is closely related to the ADMM algorithms, which are the focus of this paper, are the Augmented Lagrangian 
methods~\cite{Hes69},\cite{Pow69},  Though originally developed for general nonlinear programs, Delbos and Gilbert~\cite{DelJch05} proved global linear convergence of the algorithm on convex QPs that are feasible. More recently, Chiche and Gilbert~\cite{ChiJch14} have extended the approach to handle infeasible QPs by adaptively determining the smallest shift in 
constraints that renders the problems feasible.

In the following we provide a brief survey of recent developments in Alternating Direction Method of Multipliers (ADMM).  
ADMM has emerged as a popular optimization algorithm for the solution of structured convex programs in the areas of 
compressed sensing \cite{YanZha11}, image processing \cite{WanYanYin08}, machine learning 
\cite{ForCanGia10}, distributed optimization \cite{WeiOzd13}, regularized estimation 
\cite{WahBoyAnn12} and semidefinite programming \cite{MalPovRen09,WenGolYin10}, among 
others. ADMM algorithms were first proposed by Gabay and Mercier \cite{GabMer76} for the 
solution of variational inequalities that arise in solving partial differential equations and were 
developed in the 1970's in the context of optimization. ADMM is a special case of the Douglas-Rachford~\cite{DouRac56} splitting method,
which itself may be viewed as an instance of the proximal point
algorithm \cite{EckBer92,Roc76}. An excellent introduction to the ADMM 
algorithm, its applications, and the vast literature covering the convergence results is provided in~
\cite{Boyadmm}. 

Under  mild assumptions ADMM can be shown to converge for all choices of the 
step-size~\cite{Boyadmm}.   There have been a number of results on the global and local linear 
convergence rates of ADMM for a variety of problem settings.  Goldfarb and Ma~\cite{GolMa12} 
established that for a Jacobi version of ADMM, under assumption of Lipschitz continuous gradients, the objective 
value decreases at the rate of $O(1/k)$ and for an accelerated version at a rate of $O(1/k^2)$.  
Subsequently,~\cite{GolMaSch10} established similar rates for a Gauss-Seidel version while 
relaxing the requirement of strict convexity of both terms in the objective function.  
Deng and Yin~\cite{DenYin12} show global linear convergence under the assumption of 
strict convexity of one of the two objective functions and certain rank assumptions on the matrices 
in the coupling constraints which do not hold for~\eqref{qpform}.  He and Yuan~\cite{HeYua12}  established $O(1/k)$ convergence rates for ADMM using a variational inequality 
formulation.  The proof technique in~\cite{HeYua12} can be directly applied to~\eqref{qpform} 
to establish $O(1/k)$ rate of convergence.  However, no local convergence rates are derived.  
Hong and Luo~\cite{HonLuo13} also establish 
linear rate of convergence for ADMM under the assumption that the objective function takes a 
certain form of a strictly convex function and the step size for updating multiplier is 
sufficiently small. ADMM applied to a linear program was shown to 
converge at a global linear rate in~\cite{EckBer90}.  Boley~\cite{Bol14} analyzed the local 
rate of convergence for convex QPs~\eqref{qpform} with non-negativity under the assumption 
of an unique primal-dual solution and satisfaction of strict complementarity using a matrix recurrence technique.  
In~\cite{GhaTeiSha13}, the authors consider strictly convex QP with general inequality constraints 
which satisfy full row rank and establish global Q-linear rate of convergence using the matrix 
recurrence techniques of~\cite{Bol14} and also proposed an optimal ADMM parameter selection 
strategy.  This work was extended in~\cite{RagSte13ACC} where the authors relaxed the full row 
rank of inequality constraints and also proposed optimal ADMM parameter selection. However, 
the approach of~\cite{RagSte13ACC} results in general projection problems that are expensive to 
solve.  The work in~\cite{RagSte14MTNS} considered the solution of the QP in~\eqref{qpform}.  
However the proofs in that paper are incomplete and it does not prove that the convergence rate is 
bounded is strictly below 1.  The current paper presents an entirely new line of analysis to prove 
the same claims as in~\cite{RagSte14MTNS}.  Infeasibility detection in ADMM applied to QPs was described in a conference version by~\cite{RagSte14CDC} with sketches of the proofs.  This paper is meant to provide a comprehensive treatment of the initial developments.   

\subsection{Focus of this Work}\label{sec:ourwork}

In this work, we consider an ADMM formulation that alternates between solving an equality constrained QP and a projection on bound constraints. In particular we consider the following 
modification of the QP~\eqref{qpform},
\beq\bal
\min\limits_{\bfy,\bfw} \;& \half{\bfy^T\bfQ\bfy} + \bfq^T\bfy \\
\text{s.t.} \;& \bfA\bfy = \bfb, \bfw \in \bfYset\\
\;& \bfy = \bfw. 
\eal\label{qpformsplit}\eeq
In~\eqref{qpformsplit}, the equalities and inequalities involve separate variables, coupled by the constraint $\bfy = \bfw$. The augmented Lagrangian is defined as,
\[
L(\bfy,\bfw,\bflambda) := \half{\bfy^T\bfQ\bfy} + \bfq^T\bfy 
+\frac{\beta}{2}\left\|\bfy -\bfw - {\bflambda} \right\|^2 - \frac{\beta}{2}\|\bflambda\|^2
\]
where,  $\beta > 0$ is the ADMM parameter and we have used scaled multipliers $\beta\bflambda$ for the coupling constraints. 
 The ADMM iterations for~\eqref{qpformsplit} produces a sequence $\{\stackthr{\bfy^k}{\bfw^k}{\bflambda^k}\}$, where $\bfy^k$ always satisfies the equality constraints, $\bfw^k$ always lies within the bounds and $\beta\bflambda^k$ is the multiplier for 
bound constraints. Further, the ADMM parameter is kept fixed during the iterations. The advantage of this is that the ADMM iterations do not involve any matrix factorizations.  This 
results in simple iterations involving only matrix-vector products that can be easily implemented even in low computing power micro-controllers. 
When~\eqref{qpform} is feasible, we derive an upper bound on the rate of convergence under assumptions of:
\begin{itemize}
\item positive definiteness of the Hessian of the objective function projected on the null space of the equality constraints (reduced Hessian) 
\item linear independence constraint qualification (LICQ) holding at the solution. 
\end{itemize}
Let $\bfy^*$ denote the optimal solution to QP~\eqref{qpform}. 
We provide an explicit characterization of the rate of convergence in terms of 
\begin{itemize}
\item the eigenvalues of the reduced Hessian, 
\item the cosine of the Friedrichs 
angle~\cite{Deu01} (see Definition~9.4) between the subspace defined by the linear 
constraints and the subspace spanned by the gradients of active bound indices $\activeset^* = \{ i \;|\; \bfy^*_i = \underline{\bfy}_i \text{ or } \bfy^*_i = \overline{\bfy}_i\}$ and 
\item the ratio of the smallest distance from the bounds for inactive components to the distance from the solution, that is
\[
\frac{\min\limits_{i \notin \activeset^*} \; \min(\bfy^*_i-\underline{\bfy}_i,\overline{\bfy}_i-\bfy^*_i)}{\text{distance to the solution at current iterate}}.
\]  
\end{itemize}
Note that we do not require \emph{strict complementarity} to hold at the solution.  Active-set algorithms aim at the correct identification of the active components since they only work with a subset of the inequality constraints.  On the other hand, ADMM works with the entire set of inequalities.  Once the inactive components are correctly identified the analysis shows that positive definiteness of reduced Hessian and LICQ are sufficient to guarantee a rate of convergence that is bounded away from $1$.  When the inactive components are not correctly identified then 
there exists at least one $i \notin \activeset^*$ for which $\bfy^*_i = \underline{\bfy}_i$ { or } $\bfy^*_i =\overline{\bfy}_i$.  We exploit this to bound 
certain quantities in the analysis to yield a rate of convergence that is again bounded away from $1$.  Combining these observations yields the global Q-linear convergence result.  
In particular, the analysis shows that the iterates exhibit a two-scale rate of convergence - a 
slower rate of convergence for iterations prior to identification of the inactive indices at the 
solution  and a better rate of convergence once the inactive indices at the solution are identified.  
Numerical experiments are provided validating the theoretical analysis.

In the case of infeasible QPs, we show that 
the sequence of primal iterates $\{\stacktwo{\bfy^k}{\bfw^k}\}$ generated by ADMM converges to $\stacktwo{\bfy^\circ}{\bfw^\circ}$ with $\bfA\bfy^\circ = \bfb$, $\bfw^\circ \in \bfYset$ and $\|\bfy^\circ-\bfw^\circ\|$ is the minimum Euclidean distance between the hyperplane defined by the linear equality constraints and the convex set defined by the bounds.  
Further, we show that the sequence of multipliers $\{\bflambda^k\}$ diverges and that the divergence is restricted to a direction that lies in the range space of the equality constraints, in particular $\bfw^\circ-\bfy^\circ$.  Based on this analysis, we also propose a termination condition that recognizes when QP~\eqref{qpform} is infeasible.  

The outline of the paper is as follows.  \S\ref{sec:background} states relevant background including the assumptions, optimality conditions and infeasibility minimizer for the QP.  The ADMM formulation that we consider in the paper is presented in \S\ref{sec:admmform} and also states some properties of the ADMM iterates.  
The one-step rate of convergence analysis for feasible QPs is described in \S\ref{sec:onestepconvergence}. Q-linear convergence results for feasible QPs are provided in \S\ref{sec:convergencefeasibleqp}. \S\ref{sec:infeasqp} derives the results on the ADMM iterates when QP~\eqref{qpform} is infeasible. \S\ref{sec:numerical} presents numerical results and conclusions are provided in \S\ref{sec:conclusions}.

\subsection{Notation}

We explain the notation used in the rest of the paper below.
\begin{itemize}
\item $\R, \R_+$ denote the set of reals and set of non-negative reals,  respectively.  
$[a,b]$ denotes the closed interval $\{ x | a \leq x \leq b\}$ and $]a,b[$ denotes the 
open interval $\{x | a < x < b\}$. 
\item All vectors are assumed to be column vectors.  
For a vector $x \in \R^n$, $x^T$ denotes its transpose. 
The notation $\stacktwo{x}{y} := (x^T \; y^T)^T$ denotes the vertical stacking of the 
individual vectors.
\item $\bfe_{i}$ denotes the unit vector with $1$ at $i$-th component and $0$ elsewhere.  
\item $\symmats^n$ denotes the set of symmetric $n \times n$ matrices. 
For a matrix $X \in \symmats^{n}$,  $\lambda_i(X)$ for $i = 1,\ldots,n$ denote the 
eigenvalues of 
$X$ and $\lambda_{\text{min}}(X), \lambda_{\text{max}}(X)$ denote the minimum and maximum eigenvalues of $X$, respectively.
For a matrix $X \in \mathbb{S}^n$,  $X\succ 0$ $(X \succeq0)$ denotes 
matrix positive (semi)definiteness. For such a positive semidefinite matrix, 
$\|x\|^2_{X}$ denotes $x^TXx$. 
\item For a convex set $\bfYset \subseteq \R^n$, $\bfI_{\bfYset}(x)$ is the indicator function taking a value of $0$ for $x \in \bfYset$ and $+\infty$ otherwise.
\item For a convex set $\bfYset \subseteq \R^n$, $\proj_{\bfYset}(x)$ denotes the projection of 
$x$ onto the set $\bfYset$. $\proj_{\bfYset}$ is also called the \emph{proximal operator} of 
$\mathbb{I}_{\bfYset}$.  
For $X \in \R^{n\times n}$, $X\proj_{\bfYset}(x)$ denotes the  
product of matrix $X$ and result of the projection. 
\item We denote by $\bfI_n \in \R^{n \times n}$ the identity matrix, and 
$(\proj_{\bfYset}-\bfI_n)(x)$ denotes $\proj_{\bfYset}(x)-x$. 
\item We denote by $\refl_{\bfYset}(x)$ the reflection operator, 
$\refl_{\bfYset}(x) := (2\proj_{\bfYset} - \bfI_n)(x)$. It is also called the 
\emph{reflected proximal operator} of $\mathbb{I}_{\bfYset}$. 
\item The notation $\lambda \perp x \in \bfYset$ denotes the inequality 
$\lambda^T(x'-x) \geq 0,\forall x' \in \bfYset$, which is also called a 
{\em variational inequality}. 
\item We use $\|\cdot\|$ to denote the $2$-norm for vectors and matrices. 
\item A sequence $\{x^k\} \subset \R^n$ converging to $x^*$ is said to converge at 
\emph{Q-linear rate} if $\|x^{k+1} -x^*\| \leq \kappa\|x^k-x^*\|$ where $0 < \kappa < 1$.  
The sequence is said to converge at a R-linear rate if $\|x^k\ -x^*\| \leq \kappa^k$ for some 
$\{\kappa^k\} \rightarrow 0$ Q-linearly.
\item For the constraint matrix $\bfA \in \R^{m \times n}$ of full for row rank, 
$\bfR \in \R^{n \times m}, \bfZ \in \R^{n \times (n-m)}$ denote an orthonormal basis for the 
range space of $\bfA^T$, orthonormal basis for the null space of $\bfA$ respectively.  
\item For a $\beta > 0$, denote by 
\beq\bal
			&\;	\bfM :=\bfZ\left(\bfZ^T(\bfQ/\beta + \bfI_n)\bfZ\right)^{-1}\bfZ^T, 
\bfN :=(\bfI_n - \bfM\bfQ/\beta)\bfR(\bfA\bfR)^{-1}  \\
\text{and } 	&\;	\bfM_{\bfZ} := 2(\bfZ^T\bfQ\bfZ/\beta + \bfI_{n-m})^{-1} - \bfI_{n-m} = 
\bfZ^T(2\bfM - \bfI_n)\bfZ. 
\label{defMNMZ}
\eal\eeq
\item $\bfy, \bfw, \bflambda$ refer to ADMM variables that lies on equality constraints, 
in the constraint set and multipliers for the set inclusion constraints respectively.  
\item  $\bfy^*$ denotes optimal solution to QP~\eqref{qpform} and $\bflambda^*$ the 
optimal multipliers for set inclusion constraints. 
\item $\bfv$ refers to the Douglas-Rachford iterate and $\bfv^*$ is the fixed point of the 
Douglas-Rachford iteration.
\item $\bfu = \refl_{\bfYset}(\bfv)$ is the reflection of the Douglas-Rachford iterate $\bfv$.  
Similarly, $\bfu^* = \refl_{\bfYset}(\bfv^*)$ is the reflection of the fixed point. 
\item $\activeset^*$ denotes the indices of $\bfy^*$ that are at a bound i.e. 
$\activeset^* = \{ i \; | \; \bfy^*_i = \underline{\bfy}_i \text{ or } 
\overline{\bfy}_i \}$.  
\item $\activeset^k \subseteq \{1,\ldots,n\}$ denotes the active-set 
at iteration $k$.  
\item $\bfE^*$ and $\bfE^k$ denote the matrices of gradients of bound constraints that are 
active at the solution and at an iteration respectively.  
\item $c_F^*, c_F^k$ denotes the cosine of the Friedrich's angle between $\bfR$ and 
the subspaces spanned respectively by $\bfE^*, \bfE^k$.
\item $\delta(\|\bfM_{\bfZ}\|,c_F,\alpha^{\max})$ denotes the Q-linear convergence rate for the  
Douglas-Rachford iteration.
\end{itemize}

\section{Background} \label{sec:background}
In this section, we list the main assumptions and some key properties that are used in the 
rest of the paper.
\subsection{Main Assumptions}
We make the following assumptions on the QP~\eqref{qpform} throughout the paper.
\begin{assumption}\label{ass:feas} 
The set $\bfYset \neq \emptyset$ is non-empty.
\end{assumption}
\begin{assumption}\label{ass:rankA} The matrix $\bfA \in \R^{m\times n}$ has full row rank of $m$. 
\end{assumption}
\begin{assumption}\label{ass:redhesspd} The Hessian of objective function in~\eqref{qpform} 
is positive definite on the null space of the equality constraints, i.e., $\bfZ^T\bfQ\bfZ \succ 0$.
\end{assumption}

In subsequent sections, we make further assumptions on feasibility and linear independence of active constraint gradients at a solution of~\eqref{qpform}.

\subsection{Range and Null Spaces}
The orthonormal range and null space bases matrices $\bfR,\bfZ$ satisfy
\begin{subequations}\label{RZprops}
\beqar
\bfR^T\bfR = \bfI_m,\;\; \bfZ^T\bfZ & = & \bfI_{n-m}, \label{RZorthonormal} \\
\bfR^T\bfZ & = & 0 \label{RZorthogonal} \\
\bfR\bfR^T + \bfZ\bfZ^T & = & \bfI_n. \label{RZfullspace}
\eeqar
\end{subequations}
where~\eqref{RZorthogonal} follows from orthogonality of the range and null spaces, and~\eqref{RZfullspace} holds since $\begin{bmatrix}\bfR & \bfZ\end{bmatrix}$ is a basis for $\R^n$.

\subsection{Projection onto a Convex Set}
Given a convex set $\bfYset \subseteq \R^n$ we denote by $\proj_{\bfYset} : \R^n \rightarrow \bfYset$ the projection operator can be defined in two equivalent ways as,
\beq
\proj_{\bfYset}(\bfy) := \arg\min\limits_{\bfw \in \bfYset} \; \half \|\bfy - \bfw\|^2 \equiv 
\proj_{\bfYset}(\bfy) := \arg\min\limits_{\bfw} \mathbb{I}_{\bfYset}(\bfw) + \half \|\bfy - \bfw\|^2
\label{defprojoperator}\eeq
Thus, $\proj_{\bfYset} = (\mathbb{I}_{\bfYset} + \bfI_n)^{-1}$ 
is known as the \emph{proximal operator} of $\mathbb{I}_{\bfYset}$ and 
$\refl_{\bfYset}(\bfw) = (2\proj_{\bfYset}-\bfI_n)(\bfw)$ as the 
\emph{reflected proximal operator} of $\mathbb{I}_{\bfYset}$.
The operator $\proj_{\bfYset}$ satisfies the variational inequality,
\beq\left.\bal
\proj_{\bfYset}(\bfy) - \bfy - \bflambda = 0 \\
\bflambda \perp \proj_{\bfYset}(\bfy) \in \bfYset\eal\right\} \implies
(\proj_{\bfYset}(\bfy) - \bfy) \perp \proj_{\bfYset}(\bfy) \in \bfYset. 
\label{defprojasvi}\eeq
For all $\bfv,\;\bfv' \in \R^n$, the operators $\proj_{\bfYset}$ and $\refl_{\bfYset}$ satisfy 
(see for example,~\cite{Roc76}) :
\begin{subequations}\label{lemm:projops}
\beqar
(\proj_{\bfYset}(\bfv)-\proj_{\bfYset}(\bfv'))^T((\bfI_n-\proj_{\bfYset})(\bfv)-(\bfI_n-\proj_{\bfYset})(\bfv')) & \geq & 0 \label{lemm:projineq} \\
\|\stacktwo{\proj_{\bfYset}(\bfv)}{(\bfI_n-\proj_{\bfYset})(\bfv)} - 
\stacktwo{\proj_{\bfYset}(\bfv')}{(\bfI_n-\proj_{\bfYset})(\bfv')}\| & \leq & \|\bfv-\bfv'\| 
\label{lemm:projnonexpansive} \\
\|\refl_{\bfYset}(\bfv) - \refl_{\bfYset}(\bfv')\| & \leq & \|\bfv-\bfv'\| 
\label{lemm:reflineq}
\eeqar
\end{subequations}
where~\eqref{lemm:projineq},~\eqref{lemm:projnonexpansive} are 
called the \emph{firm-nonexpansiveness} of the proximal operator and~\eqref{lemm:reflineq} 
is called the \emph{non-expansiveness} of the reflected proximal operator.

\subsection{Optimality Conditions for QP}
We state below the optimality conditions~\cite{BoydVandenbergheBook} of QP~\eqref{qpform}.
The point $\bfy^*$ is an optimal solution of QP in~\eqref{qpform} if and only if 
there exist multipliers $\bfxi^* \in \R^m$ and $\bflambda^* \in \R^n$  satisfying,
\beq\bal
\bfQ\bfy^* + \bfA^T\bfxi^* - \bflambda^* &= -\bfq \\
\bfA\bfy^* &= \bfb \\
\bflambda^* &\perp \bfy^* \in \bfYset.
\eal\label{qpformstatconds}\eeq
We also refer to $\stackthr{\bfy^*}{\bfxi^*}{\bflambda^*}$ as a KKT point of~\eqref{qpform}. 
We denote by $\activeset^*$ the set of indices of $\bfy^*$ that lie at the bound 
\beq
\activeset^* = \{ i \;|\; \bfy^*_i = \underline{\bfy}_i \text{ or } \overline{\bfy}_i \}.
\eeq
Further, we denote by $\bfE^* \in \R^{n \times |\activeset^*|}$ the matrix corresponding to the gradients of the active bound constraints.  In other words,
\beq
\bfE^* = \begin{bmatrix} \bfe_{i_1} & \cdots & \bfe_{i_n} \end{bmatrix} \text{ with } \{i_1,\ldots,i_n\} = \activeset^*.
\label{defEstar}
\eeq

\subsection{Infeasible QP} 
Suppose Assumptions~\ref{ass:feas} and~\ref{ass:rankA} hold. Then, the QP in~\eqref{qpform} is infeasible if and only if 
\beq
\{\bfy \;|\; \bfA\bfy = \bfb \} \cap \bfYset = \emptyset. \label{nofeaspt}
\eeq
Further, there exists $\bfy^\circ$ feasible with respect to the 
linear constraints, and $\bfw^\circ \in \bfYset$, $\bfy^\circ \neq \bfw^\circ$  
satisfying,
\beq\bal
\stacktwo{\bfy^\circ}{\bfw^\circ} = \arg\min\limits_{\bfy,\bfw} &\; \half\|\bfy - \bfw\|^2 \\
\text{s.t.} &\; \bfA\bfy=\bfb,\; \bfw \in \bfYset.
\eal\label{infeasmin}\eeq
We refer to $\stackthr{\bfy^\circ}{\bfw^\circ}{\bflambda^\circ}$ as a KKT point of~\eqref{infeasmin}. 
It is easily seen from the optimality conditions of~\eqref{infeasmin} that 
\beq \left. \bal
\bfy^\circ-\bfw^\circ &\in\; \text{range}(\bfR) \\
\bfw^\circ - \bfy^\circ - \bflambda^{\circ} &=\; 0 \\
\bflambda^\circ &\perp\; \bfw^\circ \in \bfYset \eal\right\}
\implies 
\left\{ \bal \bfy^\circ-\bfw^\circ, \bflambda^\circ &\in\; \text{range}(\bfR)  \\
\bfw^\circ-\bfy^\circ &\perp\; \bfw^\circ \in \bfYset.
\eal\right.\label{diffywcirc}
\eeq 
Further, $\{\bfy \;|\; \bfA\bfy = \alpha\bfb+(1-\alpha)\bfA\bfw^\circ\}$, for any  $0 < \alpha < 1$ 
is a hyperplane separating the linear subspace defined by the equality constraints and the set $\bfYset$.  

\section{ADMM \& DR Formulation}\label{sec:admmform}

The steps of the ADMM iteration~\cite{Boyadmm} as applied to the formulation in~\eqref{qpformsplit} are:
\begin{subequations}
\beqar
\bfy^{k+1} 
&=&\; \bfM (\bfw^k + \bflambda^k - \tilde{\bfq}) + \bfN\bfb 
\label{admmiter_ext_y}\\
\bfw^{k+1} 
&=&\; \proj_{\bfYset}(\bfy^{k+1}-{\bflambda}^k) 
\label{admmiter_ext_w}\\
\bflambda^{k+1} &=&\; \bflambda^k + \bfw^{k+1}-\bfy^{k+1}\quad\quad\quad\quad\quad\quad
\label{admmiter_ext_lam}
\eeqar
\label{admmiter_ext}
\end{subequations}
where $\bfM,\bfN$ are as defined in~\eqref{defMNMZ}, and $\tilde{\bfq} = \bfq/\beta$. We can further eliminate $\bfy^{k+1}$ in \eqref{admmiter_ext} and obtain the iterations 
in condensed form as,
\beq\bal
\bfw^{k+1} =&\; \proj_{\bfYset}(\bfv^k) \\
{\bflambda}^{k+1} =&\; (\proj_{\bfYset}-\bfI_n)( \bfv^k)
\eal\label{admmiter}\eeq
where 
\beq
\bfv^k = \bfy^{k+1}-\bflambda^k = \bfM\bfw^{k}+(\bfM-\bfI_n){\bflambda}^k - \bfM\tilde{\bfq} + 
\bfN\bfb.\label{defvk}
\eeq
We can equivalently cast the ADMM iterations in~\eqref{admmiter} as iterations of the 
Douglas-Rachford (DR) method~\cite{EckBer92} as,
\beq\bal
\bfv^{k+1} 
		=&\; \half \left( (2\bfM-\bfI_n)\refl_{\bfYset}(\bfv^k) + \bfv^k \right) - \bfM\tilde{\bfq} + \bfN\bfb \\
		=&\; \half \left( (2\bfM-\bfI_n)\bfu^k + \bfv^k \right) - \bfM\tilde{\bfq} + \bfN\bfb.
\eal\label{drform}\eeq
where,
\beq
\bfu^k = \refl_{\bfYset}(\bfv^k). \label{defuk}
\eeq
We list some key properties relating ADMM and DR iterates that are 
used subsequently in the analysis.  These follow from the definitions 
in~\eqref{admmiter},~\eqref{defvk} and~\eqref{defuk}.
\beq\bal
\bfv^{k} 				=&\; \bfy^{k+1} - \bflambda^k &= \bfw^{k+1}-\bflambda^{k+1} \\
\bfu^{k} 			=&\; \bfw^{k+1}+\bflambda^{k+1} & \\
\bfu^{k}+\bfv^{k} 	=&\; 2\bfw^{k+1} & \\
-\bfu^{k}+\bfv^{k} 	=&\; -2\bflambda^{k+1}. &
\eal\label{relateadmmdr}\eeq

\subsection{Equivalence between Minima and Fixed points}
We state the equivalence between the minima of the QP in~\eqref{qpform} and the fixed 
points of the ADMM and DR iterations.  We omit the proof and refer the interested reader  
to~Lemma 25.1 in~\cite{BauCombook} to a proof based on operators or to Theorem~1 in~\cite{RagSte14MTNS}.
\begin{lemma}\label{thm:fixpointsminima}
Suppose the QP in~\eqref{qpform} an optimal solution $\bfy^*$ with multiplier 
$\bflambda^*$ for the set inclusion constraints.  
Then for any $\beta >0$, (a) $(\bfy^*,\bfy^*,\bflambda^*/\beta)$ is a fixed point of~\eqref{admmiter_ext} and (b) $\bfv^*$ is a fixed point of  the Douglas-Rachford iteration in~\eqref{drform} with $\bfu^* = \refl_{\bfYset}(\bfv^*)$ 
\beq
\bfv^* = \bfy^* - \bflambda^*
= \half \left( (2\bfM-\bfI_n)\bfu^* + \bfv^* \right) - \bfM\tilde{\bfq} + \bfN\bfb.
\label{defvstar}
\eeq 
\end{lemma}

\subsection{Results on ADMM and DR Iterates}

In the following we state some key properties of the ADMM iterates that are used for the analysis in the subsequent sections. The first result shows that at every iteration of the ADMM algorithm the variational 
inequality in \eqref{qpformstatconds} holds between $\bfw^{k+1}$ and ${\bflambda}^{k+1}$. 
\begin{lemma}\label{lemm:varineqholds}
At every iteration of the ADMM algorithm $\bfw^{k+1},{\bflambda}^{k+1}$ in 
\eqref{admmiter_ext} satisfy 
$\bfw^{k+1}\in \bfYset \perp \bflambda^{k+1}$.
\end{lemma}
\begin{proof} The updates for $\bfw^{k+1}, \bflambda^{k+1}$ are precisely of the form in~\eqref{defprojasvi} and hence, the claim holds.
\end{proof} 

The following lemma states the properties of $\bfM, \bfM_{\bfZ}$ defined in~\eqref{defMNMZ}.
\begin{lemma}\label{lemm:spectralMbnd}
Suppose Assumptions \ref{ass:rankA} and \ref{ass:redhesspd} hold. Then, 
$0 \preceq \bfM \prec \bfI_n$, and $-\bfI_{n-m} \prec \bfM_{\bfZ} \prec \bfI_{n-m}$.
\end{lemma}
\begin{proof}
From~\eqref{defMNMZ}, the eigenvalues of $\bfZ^T\bfM\bfZ$ are given by 
$(\lambda_i(\bfZ^T\bfQ\bfZ)/\beta + 1)^{-1}$. Since $\beta >0$ and $\bfZ^T\bfQ\bfZ \succ 0$ by Assumption \ref{ass:redhesspd} we have that,
 $0 < (\lambda_i(\bfZ^T\bfQ\bfZ)/\beta + 1)^{-1} < 1 
\implies 0 \prec \bfZ^T\bfM\bfZ \prec \bfI_{n-m}$.
Since $\bfR^T\bfM\bfR = 0$ 
by~\eqref{RZorthogonal} we have that $0 \preceq \bfM \prec \bfI_n$ proving the first claim.  
From the definition of $\bfM_{\bfZ}$ in~\eqref{defMNMZ} and the first claim we have that 
the second claim holds as well.
\end{proof}

\begin{lemma}\label{lemm:itervwnonexpansive} Suppose that 
$\stacktwo{\bfw^{k}}{{\bflambda}^{k}}$, $\stacktwo{\bfw^{j}}{\bflambda^{j}}$ be iterates produced by 
\eqref{admmiter}. Then, 
\beq
\|\bfv^k-\bfv^j\| \leq \|\stacktwo{\bfw^k}{{\bflambda}^k} - \stacktwo{\bfw^j}{{\bflambda}^j}\|. \label{ubnddiffv}
\eeq
\end{lemma}
\begin{proof}
Squaring the left hand side of~\eqref{ubnddiffv},
\beq\bal
&\; \|\bfv^k-\bfv^j\|^2  \\
=&\; \|\bfM(\bfw^k-\bfw^j) - (\bfI_n-\bfM)({\bflambda}^k-{\bflambda}^j)\|^2 \\
=&\; \|\bfw^k-\bfw^j\|^2_{\bfM^2} + \|{\bflambda}^k-{\bflambda}^j\|^2_{(\bfI_n-\bfM)^2}  - 2(\bfw^k-\bfw^j)^T\bfM (\bfI_n-\bfM)({\bflambda}^k-{\bflambda}^j) \\
\leq&\;\|\bfw^k-\bfw^j\|^2_{\bfM^2} + \|{\bflambda}^k-{\bflambda}^j\|^2_{(\bfI_n-\bfM)^2}  
+ \|\bfw^k-\bfw^j\|^2_{\bfM (\bfI_n-\bfM)} + \|{\bflambda}^k-{\bflambda}^j\|^2_{\bfM (\bfI_n-\bfM)} \\
\leq&\; \|\bfw^k-\bfw^j\|^2_{\bfM}+\|{\bflambda}^k-{\bflambda}^j\|^2_{(\bfI_n-\bfM)} \\
{\leq}&\; \|\bfw^k-\bfw^j\|^2 + \|{\bflambda}^k-{\bflambda}^j\|^2 
\eal\label{ubnddiffv1}\eeq
where the equality is from~\eqref{defvk}, the second equality is a simple expansion of the terms.  
The first inequality follows from
\[\bal
&\;\bfM(\bfI_n-\bfM) \succeq 0 &&\text{   (since $0 \preceq \bfM \prec \bfI_n$ by Lemma~\ref{lemm:spectralMbnd})}\\
\implies&\; \|\bfw^k-\bfw^j + ({\bflambda}^k-{\bflambda}^j)\|^2_{\bfM(\bfI_n-\bfM)} &&\geq 0 \\
\implies&\;  -2(\bfw^k-\bfw^j)^T\bfM (\bfI_n-\bfM)({\bflambda}^k-{\bflambda}^j) &&\leq \|\bfw^k-\bfw^j\|^2_{\bfM (\bfI_n-\bfM)} \\ &\; &&\;\;\;\; 
+ \|{\bflambda}^k-{\bflambda}^j\|^2_{\bfM (\bfI_n-\bfM)}.
\eal\]
The second inequality in~\eqref{ubnddiffv1} follows by collecting terms and the final inequality holds since 
$0 \preceq \bfM \prec \bfI_{n}$ (Lemma~\ref{lemm:spectralMbnd}). Hence, the claim holds.
\end{proof}

Next, we list a number of properties satisfied by the iterates~\eqref{admmiter}.
\begin{lemma}\label{lemm:iternonexpansive} Suppose that 
$\stacktwo{\bfw^{k+1}}{{\bflambda}^{k+1}}$, $\stacktwo{\bfw^{j+1}}{\bflambda^{j+1}}$ be iterates produced by 
\eqref{admmiter} from $\stacktwo{\bfw^k}{\bflambda^k}$, $\stacktwo{\bfw^j}{\bflambda^j}$ respectively. Then, the 
following hold:
\begin{enumerate}[(i)]
\item $\|\bfv^{k+1}-\bfv^{j+1}\| \leq \|\stacktwo{\bfw^{k+1}}{\bflambda^{k+1}} - \stacktwo{\bfw^{j+1}}{\bflambda^{j+1}}\|$\label{bndv}
\item $\|\stacktwo{\bfw^{k+1}}{{\bflambda}^{k+1}}-\stacktwo{\bfw^{j+1}}{{\bflambda}^{j+1}}\| \leq  \|\bfv^{k}-\bfv^{j}\|$\label{bnditersbyv}
\item $\|\stacktwo{\bfw^{k+1}}{{\bflambda}^{k+1}}-\stacktwo{\bfw^{j+1}}{{\bflambda}^{j+1}}\| \leq  
\|\stacktwo{\bfw^{k}}{{\bflambda}^{k}}-\stacktwo{\bfw^{j}}{{\bflambda}^{j}}\|$ \label{noninciters}
\item $\|\bfv^{k+1}-\bfv^{j+1}\| \leq \|\bfv^{k}-\bfv^{j}\|$. \label{eqmvk}
\end{enumerate}
\end{lemma}
\begin{proof} 
The inequality in~\eqref{bndv} follows from Lemma~\ref{lemm:itervwnonexpansive}.  From~\eqref{admmiter} and the firm non-expansiveness property~\eqref{lemm:projnonexpansive}, 
we have that~\eqref{bnditersbyv} holds. 
The inequality in~\eqref{noninciters} is obtained by applying the result in~\eqref{bndv} to the 
right hand side of~\eqref{bnditersbyv}.  
The inequality in~\eqref{eqmvk} follows from~\eqref{bndv}-\eqref{bnditersbyv}.
\end{proof}

\section{Feasible QPs - One-Step Convergence Analysis}\label{sec:onestepconvergence}

In this section we analyze the progress of the DR iterates in~\eqref{drform} to a solution over a 
single iteration.  
In particular we analyze the rate of convergence of the sequence $\{\bfv^k-\bfv^*\}$,
\beq
\bfv^{k+1} - \bfv^* 
	= \half \left( (2\bfM - \bfI_n)(\bfu^{k}-\bfu^*) + \bfv^k - \bfv^* \right).
\label{convrate}\eeq
Introducing, $\Delta\bfu^k = \bfu^k-\bfu^*$, $\Delta\bfv^k = \bfv^k-\bfv^*$ and 
the definition of $\bfM_{\bfZ}$ in~\eqref{defMNMZ} rewrite~\eqref{convrate} as,
\beq\bal
		\Delta\bfv^{k+1} =&\; \half \left( (\bfZ\bfM_{\bfZ}\bfZ^T\Delta\bfu^k + \bfZ\bfZ^T\Delta\bfv^k) + \bfR\bfR^T(-\Delta\bfu^k+\Delta\bfv^k) \right) \\
\implies	\|\Delta\bfv^{k+1}\|^2 =&\; \frac{1}{4}\left( \|\bfM_{\bfZ}\bfZ^T\Delta\bfu^k + \bfZ^T\Delta\bfv^k\|^2 + \|\bfR^T(-\Delta\bfu^k+\Delta\bfv^k)\|^2 \right) \\
					      \leq&\; \frac{1}{4}\left( \left(\|\bfM_{\bfZ}\| \|\bfZ^T\Delta\bfu^k\| + \|\bfZ^T\Delta\bfv^k\|\right)^2 +  \|\bfR^T(-\Delta\bfu^k+\Delta\bfv^k)\|^2 \right) \\
					      	\leq&\; \frac{1}{4}\left( \left(\|\bfM_{\bfZ}\| \zeta^k_u + \zeta^k_v\right)^2 \|\Delta \bfv^k\|^2 + \|\bfR^T(-\Delta\bfu^k+\Delta\bfv^k)\|^2
					      	 \right)
\eal\label{convrate0}\eeq
 where the first inequality in~\eqref{convrate0} follows from the triangle inequality and the second inequality follows from the definition of $\zeta_u^k, \zeta_v^k$ below, 
\beq
\zeta^k_{u} = \frac{\|\bfZ^T\Delta\bfu^k\|}{\|\Delta\bfv^k\|}, \zeta^k_{v} = \frac{\|\bfZ^T\Delta\bfv^k\|}{\|\Delta\bfv^k\|}. \label{defZDeltauv}
\eeq
The right hand side in~\eqref{convrate0} can be rewritten using~\eqref{relateadmmdr} as,
\beq\bal
\|\Delta \bfv^{k+1}\|^2 \leq&\; \frac{1}{4}\left( \left(\|\bfM_{\bfZ}\| \zeta^k_u + \zeta^k_v\right)^2 \|\Delta \bfv^k\|^2 + 4\|\bfR^T(\bflambda^{k+1}-\bflambda^*)\|^2 \right) \\
\leq&\; \frac{1}{4}\left( \left(\|\bfM_{\bfZ}\| \zeta^k_u + \zeta^k_v\right)^2  + 
4(c^k \alpha^k)^2 \right)	\|\Delta \bfv^k\|^2	
\eal\label{convrate1}\eeq
\text{where}, $c^k, \alpha^k$ are defined as,
\beq
c^k = \frac{|\bfR^T(\bflambda^{k+1}-\bflambda^*)|}{\|\bflambda^{k+1}-\bflambda^*\|} \leq 1 
\text{ and } \alpha^k = \frac{\|\bflambda^{k+1}-\bflambda^*\|}{\|\Delta\bfv^k\|} \leq 1	
\label{defckalphak}\eeq
where the bound on $c^k$ follows from Cauchy-Schwarz and orthonormality of $\bfR$ 
in~\eqref{RZorthonormal} and the bound on $\alpha^k$ follows from 
Lemma~\ref{lemm:iternonexpansive}\eqref{bnditersbyv}.  

Our primary objective in this section is to show that the right hand side of~\eqref{convrate1} 
yields a contraction.  To motivate the difficulty, consider 
\[
\sup\limits_{\zeta^k_u,\zeta^k_v,\alpha^k \in [0,1]} \; \frac{1}{4}\left( \left(\|\bfM_{\bfZ}\| \zeta^k_u + \zeta^k_v\right)^2  + 
4(c^k \alpha^k)^2 \right)
\]
where the obvious bounds on the iteration related quantities, $0 \leq \zeta^k_u,\zeta^k_v,
\alpha^k \leq 1$, have been employed.   It is easy to see that the supremum of $1$ is 
attained if:
\begin{itemize}
\item $c^k = 1$ and we choose $\zeta_u^k = \zeta_v^k = 0, \alpha^k= 1$
\item $c^k < 1$, $\|\bfM_{\bfZ}\| \geq \sqrt{1 - (c^k)^2}$ and we choose  
$\zeta_u^k = \sqrt{1 - (c^k)^2}/\|\bfM_{\bfZ}\|$, $\zeta^k_v = \sqrt{1 - (c^k)^2}$, $\alpha^k  = 1$.
\end{itemize}
Thus, the obvious bounds on $\zeta^k_u,\zeta^k_v,\alpha^k$ are not sufficient 
to obtain a contraction in~\eqref{convrate1}.  In this section, we show that for~\eqref{convrate1} 
to be a contraction it is sufficient that either $\alpha^k < 1$ or $c^k < 1$. 
To do this we derive additional inequalities relating $\zeta^k_u,\zeta^k_v,\alpha^k$.  

The roadmap of the analysis is as follows.  We introduce a notion of \emph{active-set} in 
\S\ref{sec:activeset} and use this to derive a more generic bound on $c^k$ in terms of 
cosine of the Friedrich's angle between subspaces.    
The range space term is  
bounded above in \S\ref{sec:Rtermbnd}.  
\S\ref{sec:bndzetafrom0} derives lower bounds on the null space quantities. 
Finally, \S\ref{sec:worstCaseBnd} derives the worst-case convergence factor 
and shows that it is indeed $< 1$ as long as $\alpha^k < 1$ or $c^k < 1$. 

\subsection{Active-set}\label{sec:activeset}

We denote the \emph{active-set} at iterate $k$ as $\activeset^k$ and define it as,
\beq
\activeset^k = \{ i \;|\; -\bflambda^{k+1}_i + \bflambda^*_i   \neq 0\} 
\overset{\eqref{relateadmmdr}}{=} \{ i \;|\; 
-\Delta\bfu^k_i + \Delta\bfv^k_i \neq 0 \}. \label{defAk}
\eeq
Since $\bfw^{k+1},\bflambda^{k+1}$ satisfy the variational inequality 
(Lemma~\ref{lemm:varineqholds}), we must have $\bfw^{k+1}_i \in ]\underline{\bfy}_i,
\overline{\bfy}_i[$ and $\bflambda^{k+1}_i = 0$ for all $i \notin \activeset^*$.  
Thus,~\eqref{defAk} implies that 
$\{1,\ldots,n\}\setminus\activeset^* \subseteq \{1,\ldots,n\}
\setminus\activeset^k$. 
Denote by $\bfE^k \in \R^{n \times |\activeset^k|}$ a matrix that is defined as,
\beq
\bfE^k = \begin{bmatrix} \bfe_{i_1} & \cdots & \bfe_{i_p} \end{bmatrix} \text{ where } {i_j} \in \activeset^k. \label{defEk}
\eeq
From the above definition we have that, 
$\bflambda^{k+1}-\bflambda^* \in \text{range}(\bfE^k)$.  
Further, if we denote by $c_F^k$ the cosine of 
the Friedrich's angle between $\bfR$ and $\bfE^k$ we have that,
\beq
c_F^k = \|\bfR^T\bfE^k\| \geq c^k. \label{defcFk}
\eeq
Note that $c^k_F$ may not be strictly less than $1$ for all active sets.  

\subsection{Bounding the Range Space Term in~\eqref{convrate1}}\label{sec:Rtermbnd}

The range space term in~\eqref{convrate1} is bounded in two ways.  Firstly, 
from~\eqref{relateadmmdr} we obtain,
\[
\|\bfR^T(-\Delta\bfu^k+\Delta\bfv^k)\| = 2\|\bfR^T(\bflambda^{k+1}-\bflambda^*)\| 
							\leq 2c_F^k\alpha^k\|\Delta\bfv^k\|
\label{ubOnRspaceterm1}\]
where the inequality follows from~\eqref{defckalphak} and~\eqref{defcFk}.  On the other hand, we can also use the triangle inequality to obtain another upper bound as,
\[\bal
\|\bfR^T(-\Delta\bfu^k+\Delta\bfv^k)\| \leq&\; \|\bfR^T\Delta\bfu^k\| + \|\bfR^T\Delta\bfv^k\| \\
							\leq&\; \left(\frac{\|\bfR^T\Delta\bfu^k\|}{\|\Delta\bfv^k\|} + \sqrt{1-(\zeta^k_v)^2} \right) \|\Delta\bfv^k\|
\eal\label{ubOnRspaceterm2}\]
where the first inequality follows from definition of $\zeta^k_v$ in~\eqref{defZDeltauv}. 
Since $\|\Delta\bfu^k\| \leq \|\Delta\bfv^k\|$ by non-expansive property of reflected 
proximal operator~\eqref{lemm:reflineq}, 
\beq\bal
\|\Delta\bfu^k\|^2 = \|\bfR^T\Delta\bfu^k\|^2 + \|\bfZ^T\Delta\bfu^k\|^2 \leq&\; \|\Delta\bfv^k\|^2 \\
\implies \|\bfR^T\Delta\bfu^k\| \leq&\; \sqrt{1-(\zeta^k_u)^2} \|\Delta\bfv^k\|
\eal\label{bndRDeltauByDeltav}\eeq
where the second inequality follows by the definition of $\zeta^k_u$ in~\eqref{defZDeltauv} 
and taking the square root.  Hence, the range-space term 
in~\eqref{convrate1} can be bounded as,
\beq\bal
&\; 2\|\bfR^T(-\bflambda^{k+1}+\bflambda^*)\| \leq \gamma^k \|\Delta\bfv^k\| \\
\text{ where, } &\; \gamma^k = \min\left( 2c_F^k\alpha^k , \sqrt{1-(\zeta^k_u)^2} + \sqrt{1-(\zeta^k_v)^2} \right).\label{ubOnRspaceterm}
\eal\eeq

\subsection{Lower Bound on \texorpdfstring{$(\zeta^k_u+\zeta^k_v)$}{}}\label{sec:bndzetafrom0}

\begin{subequations}
From~\eqref{relateadmmdr} we have that,
\beq\bal
-\Delta\bfu^{k}+\Delta\bfv^{k} =&\; 2(-\bflambda^{k+1}+\bflambda^*) \\
\implies 
\|\bfZ^T(-\Delta\bfu^k+\Delta\bfv^k)\| =&\; 2\|\bfZ^T(-\bflambda^{k+1}+\bflambda^*)\|. 
\eal\label{lbOnzetauzetav0}\eeq
The right hand side can be lower bounded as,
\beq\bal
\|\bfR^T(\bflambda^{k+1}-\bflambda^*)\|^2 + \|\bfZ^T(\bflambda^{k+1}-\bflambda^*)\|^2 
=&\; \|\bflambda^{k+1}-\bflambda^*\|^2 \\
\implies \|\bfZ^T(\bflambda^{k+1}-\bflambda^*)\|^2 =&\; \left( 1 - (c^k)^2 \right) 
\|\bflambda^{k+1}-\bflambda^*\|^2 \\
\geq&\; (1 - (c_F^k)^2)(\alpha^k)^2 \|\Delta\bfv^k\|^2
\eal\label{lbrhsbndlb}\eeq 
where, the implication in the above follows from rearranging and substitution of $c^k$ 
in~\eqref{defckalphak}. The inequality follows from~\eqref{defcFk} and definition of $\alpha^k$ 
in~\eqref{defckalphak}. 
The left hand side in~\eqref{lbOnzetauzetav0} can be upper bounded using the triangle 
inequality as,
\beq
\|\bfZ^T(-\Delta\bfu^k+\Delta\bfv^k)\| \leq \|\bfZ^T\Delta\bfu^k\| + \|\bfZ^T\Delta\bfv^k\| = (\zeta^k_u + \zeta^k_v)\|\Delta\bfv^k\| \label{lblbsbndub}
\eeq
\end{subequations}
where the equality is by~\eqref{defZDeltauv}. Substituting~\eqref{lbrhsbndlb},~\eqref{lblbsbndub} in~\eqref{lbOnzetauzetav0}, we obtain 
\beq
\zeta^k_u + \zeta^k_v \geq 2\sqrt{1 - (c_F^k)^2}\alpha^k. \label{lbOnzetauzetav}
\eeq

\subsection{Worst-case Bound on Convergence Rate}\label{sec:worstCaseBnd}

Using the inequalities~\eqref{ubOnRspaceterm} and \eqref{lbOnzetauzetav}, we can 
define the worst-case convergence rate in~\eqref{convrate1} as
\beq\bal
\delta(\|\bfM_{\bfZ}\|,c_F,\alpha^{\max})^2 = \sup\limits_{\zeta_u,\zeta_v,\alpha,\gamma} &\; \frac{1}{4}\left( \left(\|\bfM_{\bfZ}\|\zeta_u + \zeta_v\right)^2 + \gamma^2 \right) \\
\mbox{s.t.} &\; (\zeta_u + \zeta_v)^2 \geq 4(1 - c_F^2)\alpha^2  \\
&\; \gamma^2 \leq 4c_F^2\alpha^2  \\
&\; \gamma^2 \leq (\sqrt{1-\zeta_u^2} + \sqrt{1-\zeta_v^2})^2 \\
&\; 0 \leq \zeta_u,\zeta_v \leq 1, 0 \leq \alpha  \leq \alpha^{\max} 
\eal\label{defdeltacF}\eeq
where $\alpha^{\max}$ is a parameter introduced to upper-bound $\alpha$.  Our goal in 
this subsection is to show that $\delta(\|\bfM_{\bfZ}\|,c_F,\alpha^{\max}) < 1$ when 
$c_F < 1$ or $\alpha^{\max} < 1$.  This is sufficient to show that 
$\delta(\|\bfM_{\bfZ},c_F,\alpha^{\max}) < 1$ since the feasible region in~\eqref{defdeltacF} 
is reduced when $c^F$ or $\alpha^{\max}$ is decreased.  

The supremum in~\eqref{defdeltacF} is attained since $\alpha,\zeta_u,\zeta_v$ all lie in a compact set.  Note that we have allowed for $\zeta_u$ to be in $[0,1]$ even though that might not 
necessarily happen based on the definition in~\eqref{defZDeltauv}.  The inequality in~\eqref{lbOnzetauzetav} is written as a squared inequality.  
Also, the constraint involving the ``min'' term in~\eqref{ubOnRspaceterm} is squared and replaced as two inequalities.  
The optimization problem in~\eqref{defdeltacF} is an instance of a quadratically constrained quadratic program (QCQP) and does not lend itself to easy analysis in the present form.  
To show that this is indeed a valid bound, we consider the semidefinite programming (SDP) relaxation of~\eqref{defdeltacF}.  Prior to presenting the SDP, we introduce the following 
matrix variables,
\beq
X = \begin{bmatrix} \zeta_u \\ \zeta_v \end{bmatrix} \begin{bmatrix} \zeta_u & \zeta_v \end{bmatrix}, 
Y = \begin{bmatrix} \sqrt{1-\zeta_u^2} \\ \sqrt{1-\zeta_v^2} \end{bmatrix} \begin{bmatrix} \sqrt{1-\zeta_u^2} & \sqrt{1-\zeta_v^2} \end{bmatrix} \label{defSDPmatrix}
\eeq
and data matrices
\[
C = \begin{bmatrix} \kappa \\ 1 \end{bmatrix}\begin{bmatrix} \kappa & 1 \end{bmatrix}, 
E = \begin{bmatrix} 1 \\ 1 \end{bmatrix} \begin{bmatrix} 1 & 1 \end{bmatrix}  \text{ with } 
\kappa = \|\bfM_{\bfZ}\|. 
\]
The SDP relaxation of~\eqref{defdeltacF} is,
\beq\bal
\delta_{SDP}(\|\bfM_{\bfZ}\|,c_F,\alpha^{\max})^2 = \sup\limits_{X,Y,\alpha,\gamma} &\; \frac{1}{4}\left(   C\bullet X   + \gamma^2 \right) \\
\mbox{s.t.} &\; E\bullet X  \geq 4(1 - c_F^2)\alpha^2  \\
&\; \gamma^2 \leq 4c_F^2\alpha^2  \\
&\; \gamma^2 \leq E\bullet Y \\
&\; X_{11} + Y_{11} = 1 , X_{22} + Y_{22} = 1 \\
&\; X, Y \succeq 0, 0 \leq \alpha  \leq \alpha^{\max} 
\eal\label{defdeltacFSDP}\eeq
where for $A,B \in \symmats^n$, $A\bullet B := \sum_{i=1}^n\sum_{j=1}^nA_{ij}B_{ij}$ represents the trace inner product between the matrices.  The SDP 
enforces additional constraints $X_{11}+Y_{11} = 1$, $X_{22}+Y_{22} = 1$ to enforce the relations $\zeta_u^2 + (1-\zeta_u^2) = 1$, $\zeta_v^2 + (1-\zeta_v^2) = 1$ 
respectively.  Since the SDP~\eqref{defdeltacFSDP} does not enforce the rank-1 requirement on matrices $X,Y$ this is a relaxation of~\eqref{defdeltacF}.  Hence, $\delta_{SDP}(\|\bfM_{\bfZ}\|,c_F,\alpha^{\max}) \geq \delta(\|\bfM_{\bfZ}\|,c_F,\alpha^{\max})$.  In the following we show that the objective values are in fact equal and hence, the convex SDP formulation can be used to obtain the bound in~\eqref{defdeltacF}.  
We use the proof technique of Kim and Kojima~\cite{KimKoj03} to show the result.

\begin{lemma}\label{lemm:defdeltacFSDP}
$\delta_{SDP}(\|\bfM_{\bfZ}\|,c_F,\alpha^{\max}) = \delta(\|\bfM_{\bfZ}\|,c_F,\alpha^{\max})$.
\end{lemma}
\begin{proof}
The SDP in~\eqref{defdeltacFSDP} has a compact feasible set and hence, the supremum is always attained.  If $\alpha^{\max} = 0$ then the variables $\gamma, \alpha$ can be eliminated from the problem and we have that the maximum value for the SDP occurs at $X^* = E$ with 
an objective value of $(1+\kappa)^2/4$.  Since $X^*$ has rank-1 we have that the claim holds for $\alpha^{\max} = 0$.  
We assume without loss of generality that $\alpha^{\max} > 0$. 
Since the SDP is strictly feasible ($X = \bfI_2, Y = \bfI_2$ is always feasible), strong duality holds for the SDP.  
Suppose $(X^*,Y^*,\alpha^*,\gamma^*)$ solves the SDP problem~\eqref{defdeltacFSDP}.  Define $\hat{x}, \hat{y}$ as,
\[
\hat{x} = \begin{bmatrix} \sqrt{X^*_{11}} \\ \sqrt{X^*_{22}} \end{bmatrix}, \hat{y} = \begin{bmatrix} \sqrt{Y^*_{11}} \\ \sqrt{Y^*_{22}} \end{bmatrix}.
\]
We show in the following that $(\hat{x}\hat{x}^T,\hat{y}\hat{y}^T,\alpha^*,\gamma^*)$ is feasible for the SDP.  By definition of $\hat{x},\hat{y}$ it is easy to verify that equality 
constraints in~\eqref{defdeltacFSDP} hold.  Since $X^* \succeq 0$,
\beq\bal
(X^*_{12})^2 \leq X^*_{11}X^*_{22} \implies X^*_{12} \leq  \sqrt{X^*_{11}} \sqrt{X^*_{22}} \implies E \bullet X^* \leq E \bullet \hat{x}\hat{x}^T \\
\text{Since, } E\bullet X^* \geq 4(1-c_F^2)(\alpha^*)^2 \implies E\bullet \hat{x}\hat{x}^T \geq  4(1-c_F^2)(\alpha^*)^2. 
\eal\label{rank1IsFeas}\eeq
Using identical arguments it can be shown that $E\bullet \hat{y}\hat{y}^T \geq (\gamma^*)^2$.  This proves the feasibility of 
$(\hat{x}\hat{x}^T,\hat{y}\hat{y}^T,\alpha^*,\gamma^*)$ for the SDP.  Further, since $\kappa > 0$ the arguments in~\eqref{rank1IsFeas} can be repeated for the term in the objective to obtain that $C \bullet \hat{x}\hat{x}^T \geq C \bullet X^*$. Since, the SDP is convex the it must be true that  $C \bullet \hat{x}\hat{x}^T = C \bullet X^*$.  Thus, we have constructed a 
rank-1 solution to the SDP with optimal objective value.  
This proves the claim.
\end{proof}

Lemma~\ref{lemm:defdeltacFSDP} allows us to compute the worst-case contraction factor in~\eqref{defdeltacF} through the solution of a convex program in~\eqref{defdeltacFSDP} for which efficient solvers~\cite{Sedumi} exist.  Tables~\ref{tab:deltacF} and ~\ref{tab:deltacFalpha} 
list $\delta(\|\bfM_{\bfZ}\|,c_F^k,\alpha^{\max})$ obtained using the above procedure for 
worst-case scenario of $\alpha^{\max} = 1 > c_F$ and $\alpha^{\max} < 1 = c_F$ respectively.   
From the table, it is clear that 
$\delta(\|\bfM_{\bfZ}\|,c_F^k,\alpha^{\max}) < 1$ if $c_F < 1$ or  $\alpha^{\max} < 1$.

\begin{table}[ht]
\begin{center}
\begin{tabular}{|l|c|c|c|c|c|c|}
\hline
$c_F \downarrow$ & \multicolumn{6}{c|}{$ \leftarrow \|\bfM_{\bfZ}\| \rightarrow$} \\
\cline{2-7}
& 0.000& 0.200& 0.400& 0.600& 0.800& 0.999\\
\hline
0.000 & 0.500 & 0.600 & 0.700 & 0.800 & 0.900 & 0.9995\\
0.200 & 0.537 & 0.626 & 0.717 & 0.810 & 0.904 & 0.9995\\
0.400 & 0.627 & 0.692 & 0.763 & 0.838 & 0.917 & 0.9996\\
0.600 & 0.742 & 0.784 & 0.830 & 0.882 & 0.938 & 0.9997\\
0.800 & 0.868 & 0.888 & 0.911 & 0.937 & 0.966 & 0.9998\\
0.999 & 0.9993 & 0.9994 & 0.9995 & 0.9997 & 0.9998 & $\approx 1-10^{-6}$\\
\hline
\end{tabular}
\end{center}
\caption{Numerical estimates of $\delta(\|\bfM_{\bfZ}\|,c_F,1.0)$ for different values of $\|\bfM_{\bfZ}\|$ and $c_F$.}\label{tab:deltacF}
\end{table}

\begin{table}[ht]
\begin{center}
\begin{tabular}{|l|c|c|c|c|c|c|}
\hline
$\alpha^{\max} \downarrow$ & \multicolumn{6}{c|}{$ \leftarrow \|\bfM_{\bfZ}\| \rightarrow$} \\
\cline{2-7}
& 0.000& 0.200& 0.400& 0.600& 0.800& 0.999\\
\hline
0.000 & 0.500 & 0.600 & 0.700 & 0.800 & 0.900 & 0.9995\\
0.200 & 0.539 & 0.626 & 0.717 & 0.810 & 0.904 & 0.9995\\
0.400 & 0.640 & 0.697 & 0.764 & 0.838 & 0.917 & 0.9996\\
0.600 & 0.775 & 0.795 & 0.834 & 0.883 & 0.938 & 0.9997\\
0.800 & 0.894 & 0.900 & 0.915 & 0.938 & 0.966 & 0.9998\\
0.999 & 0.9995 & 0.9995 & 0.9996 & 0.9997 & 0.9998 & $\approx 1-10^{-6}$\\
\hline
\end{tabular}
\end{center}
\caption{Numerical estimates of $\delta(\|\bfM_{\bfZ}\|,1.0,\alpha^{\max})$ for different values of $\|\bfM_{\bfZ}\|$ and $\alpha^{\max}$.}\label{tab:deltacFalpha}
\end{table}

\section{Q-Linear Convergence}\label{sec:convergencefeasibleqp}

We use the analysis in \S\ref{sec:onestepconvergence} to establish that $\{\bfv^k\}$ converges at a Q-linear rate and 
$\{(\bfw^k,\bflambda^k)\}$ converges at a 2-step Q-linear rate.  
We assume through this section that Assumptions~\ref{ass:feas}-\ref{ass:redhesspd} hold.  
In addition, we also assume that the QP has an optimal solution and that the linear 
independence constraint qualification (LICQ)~\cite{BoydVandenbergheBook} holds at the 
solution.

\begin{assumption}\label{ass:feasqp}
The QP in~\eqref{qpform} has an optimal solution $\bfy^*$ with multipliers $\bflambda^*$ for the bound constraints.
\end{assumption}

\begin{assumption}\label{ass:licq} 
The linear independence constraint qualification (LICQ) holds at the solution, that is, the 
matrix $[\bfR \;\; \bfE^*]$ is full column rank.
\end{assumption}

A consequence of LICQ is that the largest singular value of $\bfR^T\bfE^*$ is $<1$,
\beq
\|\bfR^T\bfE^*\|  = c^*_{F} < 1 \label{angleBndR}
\eeq
the cosine of the Friedrich's angle~\cite[Definition~9.4]{Deu01} 
between the subspaces spanned by vectors in $\bfR$ and $\bfE^*$.  
The rest of the section is organized as follows. \S\ref{sec:localconvergence} 
and \S\ref{sec:globalconvergence} show the Q-linear convergence when 
$\activeset^k \subseteq \activeset^*$ and $\activeset^k \nsubseteq \activeset^*$ 
respectively.  \S\ref{sec:qlinearconvergence} proves the 
Q-linear convergence result for the full sequence and \S\ref{sec:optbeta} derives the 
optimal ADMM parameter $\beta^*$.  \S\ref{sec:relatedlit} compares our results with those 
in the literature.

\subsection{Convergence Rate for \texorpdfstring{$\activeset^k \subseteq \activeset^*$}{}}\label{sec:localconvergence}
From the definition of active-set $\activeset^k$ in~\eqref{defAk}, we have that,
\beq
\activeset^k \subseteq \activeset^* \iff \underline{\bfy}_i < \bfw^{k+1}_i < \overline{\bfy}_i 
\;\forall\; i \notin \activeset^* \label{optactiveset}
\eeq
since $\bflambda^*_i= 0$ for $i \notin \activeset^*$.   Hence, eventually the ADMM iterates enter 
a neighborhood of the solution where this holds.  Note that this is guaranteed regardless 
of the assumption on strict complementarity holding at the solution. 
For all such iterates we have from Assumption~\ref{ass:licq} 
that $c^k_F \leq c_F^*$ and this yields a contraction as shown below.

\begin{theorem}\label{thm:localconvergence}
Suppose Assumptions~\ref{ass:feas}-\ref{ass:licq} hold.  Then for all iterates $k$ such that $\activeset^k \subseteq \activeset^*$, 
$\|\bfv^{k+1}-\bfv^*\| \leq \delta(\|\bfM_{\bfZ}\|,c_F^*,\alpha^k)\|\bfv^k - \bfv^*\|$ with convergence  
rate $\delta(\|\bfM_{\bfZ}\|,c_F^*,\alpha^k) < 1$ where $\alpha^k \leq 1$.
\end{theorem}
\begin{proof}
If $\activeset^k \subseteq \activeset^*$, then the columns of $\bfE^k$ are a subset of the 
columns  of $\bfE^*$.  Hence, $c^k_F \leq c_F^* < 1$ by Assumption~\ref{ass:licq}.  
The analysis in \S\ref{sec:onestepconvergence} applies to yield that 
$\delta(\|\bfM_{\bfZ}\|,c_F^*,\alpha^k) \leq \delta(\|\bfM_{\bfZ}\|,c_F^*,1.0) < 1$ 
and the claim follows.
\end{proof}

\subsection{Convergence Rate for \texorpdfstring{$\activeset^k \nsubseteq \activeset^*$}{}}
\label{sec:globalconvergence}
 
We begin by deriving a worst-case upper bound on $\alpha^k$ by varying over $\bfv$ 
such that $\|\Delta \bfv\| = \|\Delta\bfv^k\| = \Delta^k$.   
From the definition of active-set in~\eqref{defAk},
\beq\bal
\activeset^k \nsubseteq \activeset^*	\implies &\; \exists\; i \in \activeset^k \setminus \activeset^*\text{ such that }  \bflambda^{k+1}_i \neq 0 \\
\implies&\;  \exists\; i \in \activeset ^k\setminus \activeset^*\text{ such that }  \bfw^{k+1}_i  = 
\underline{\bfy}_i \text{ or }\overline{\bfy}_i.
\eal\label{noptactiveset}\eeq
We seek to obtain the supremum of the following program,
\beq\bal
\sup\limits_{\bfv} 	&\; \alpha \\
\text{s.t.}						&\; \alpha = \frac{\|\bflambda^{+}-\bflambda^*\|}{\|\Delta\bfv\|} \\
							&\; \|\Delta\bfv\| = \Delta^k,\; \activeset \nsubseteq \activeset^*
\eal\label{maxalphak}\eeq
where $\bflambda^+$ denotes the multiplier resulting from the ADMM iteration 
in~\eqref{admmiter}.  Since, the above is the supremum over all possible $\bfv$ 
satisfying $\|\Delta\bfv\| = \Delta^k$ it follows that this is an upper bound for $\alpha^k$.  
Note that $\bfv$ lies in a compact set and consequently, the supremum 
in~\eqref{maxalphak} is attained.  
We will assume without loss of generality that 
\beq
\underline{\bfy}_i \text{ or } \overline{\bfy}_i  \text{ is finite for at least one } i \notin \activeset^*.  \label{finiteBndsForiGtna}
\eeq
If this is not the case,~\eqref{noptactiveset} shows that 
$\activeset^k \nsubseteq \activeset^*$ never occurs and hence, the analysis in \S\ref{sec:localconvergence} applies.  Define, 
\beq
\Delta\bfy^*_i = \min\left( \bfy^*_i-\underline{\bfy}_i,\overline{\bfy}_i-\bfy^*_i \right) \text{ and }
i^{\min} = \arg\min_{i \notin \activeset^*}  \Delta\bfy^*_i. \label{defDeltaymin}
\eeq
The quantity $\Delta\bfy^*_{i^{\min}}$ measures the smallest distance from the bounds for the indices that are inactive at the solution.  This will play a critical role in deriving the upper bound. Further, $\Delta\bfy^*_i > 0$ exists by LICQ (Assumption~\ref{ass:licq}) and~\eqref{finiteBndsForiGtna}. 
The following lemma upper bounds $\alpha^k$.

\begin{lemma}\label{lemm:maxalphak}
Suppose Assumptions~\ref{ass:feas}-\ref{ass:licq} hold, 
$\activeset^k \nsubseteq \activeset^*$ and $\Delta^k := \|\Delta\bfv^k\|$.  Then, 
\beq
\alpha^k \leq \alpha^{\max}(\Delta^k)  \text{ where, }
\alpha^{\max}(\Delta^k) = \sqrt{1 - \left(\frac{\Delta\bfy^*_{i^{\min}}}{\Delta^k}\right)^2}. \label{defmaxalphak}
\eeq
\end{lemma}
\begin{proof}
From the definition of $\alpha^k$ in~\eqref{defckalphak} and $\bfv^k$ in~\eqref{relateadmmdr} 
we have that,
\beq\bal
(\alpha^k)^2 
	 	=&\; \frac{\|\bflambda^{k+1}-\bflambda^*\|^2}{\|\Delta\bfv^k\|^2} 
		= \frac{\|\bflambda^{k+1}-\bflambda^*\|^2}{\|\bfw^{k+1}-\bfw^*-\bflambda^{k+1}+\bflambda^*\|^2} \\
=&\; 1- \frac{2(\bfw^{k+1}-\bfw^*)^T(-\bflambda^{k+1}+\bflambda^*) + \|\bfw^{k+1}-\bfw^*\|^2}{(\Delta^k)^2} 
\eal\label{sqralphak}\eeq 
where the last expression is obtained by simplification and assumption that $\|\Delta\bfv^k\|$ is specified.  From the firm non-expansiveness property~\eqref{lemm:projineq} 
and~\eqref{relateadmmdr}, we have that 
\beq\bal
		&\; (\bfw^{k+1}-\bfw^*)^T(-\bflambda^{k+1}+\bflambda^*) \\
		=&\; \left(\proj_{\bfYset}(\bfv^k) - \proj_{\bfYset}(\bfv^*)\right)^T \left((\bfI_n-\proj_{\bfYset})(\bfv^k) - 
(\bfI_n-\proj_{\bfYset})(\bfv^*)\right) \geq 0. 
\eal\label{bndOnDeltalam}\eeq
Since $\activeset^k \nsubseteq \activeset^*$, 
there exists at least one $i' \in \activeset^k \setminus \activeset^*$ satisfying~\eqref{noptactiveset}. 
Obviously, the maximum in~\eqref{sqralphak} occurs for $\bfw^{k+1},\bflambda^{k+1}$ 
such that numerator is as small as possible.  In view of~\eqref{bndOnDeltalam} 
and~\eqref{defDeltaymin} this occurs for
\[
(\bfw^{k+1}-\bfw^*)^T(-\bflambda^{k+1}+\bflambda^*) = 0 
\text{ and } \|\bfw^{k+1}-\bfw^*\|^2 =  (\Delta\bfy^*_{i^{\min}})^2.
\]
 The existence of 
 $\Delta\bfy^*_{i^{\min}} > 0$ is guaranteed by Assumption~\ref{ass:licq} 
 and~\eqref{finiteBndsForiGtna}.  Hence, $(\alpha^k)^2 \leq {1 - \left(\frac{\Delta\bfy^*_{i^{\min}}}
 {\Delta^k}\right)^2}$ completing the proof.
\end{proof}

We can now state the result on the convergence rate for $\activeset^k \nsubseteq \activeset^*$.
\begin{theorem}\label{thm:globalconvergence}
Suppose Assumptions~\ref{ass:feas}-\ref{ass:licq} hold. Then for all iterates $k$ such that 
$\activeset^k \nsubseteq \activeset^*$, $\|\bfv^{k+1}-\bfv^*\| \leq 
\delta(\|\bfM_{\bfZ}\|,c_F^k,\alpha^{\max}(\Delta^k))\|\bfv^k - \bfv^*\|$, with convergence rate 
$\delta(\|\bfM_{\bfZ}\|,c_F^k,\alpha^{\max}(\Delta^k)) < 1$ where $\Delta^k = \|\Delta\bfv^k\|$, $c_F^k \leq 1$. 
\end{theorem}
\begin{proof}
Since $\activeset^k \nsubseteq \activeset^*$, Lemma~\ref{lemm:maxalphak} guarantees 
the existence of $\alpha^{\max}(\Delta\bfv^k) < 1$.  The analysis in 
\S\ref{sec:onestepconvergence} applies to yield  that 
$\delta(\|\bfM_{\bfZ}\|,1.0,\alpha^{\max}(\Delta^k)) < 1$.  Since 
$\delta(\|\bfM_{\bfZ}\|,c_F^k,\alpha^{\max}(\Delta^k)) \leq \delta(\|\bfM_{\bfZ}\|,1.0,\alpha^{\max}(\Delta^k))$  we have the said result.
\end{proof}

\subsection{Q-Linear Convergence}\label{sec:qlinearconvergence}
We prove the Q-linear convergence result below.
\begin{theorem}\label{thm:qlinearconvergence}
Suppose Assumptions~\ref{ass:feas}-\ref{ass:licq} hold. 
Let $\bfw^0,\bflambda^0$ be the initial iterates for the ADMM iteration in~\eqref{admmiter_ext}.  Then,
\beq
			\|\Delta\bfv^{k+1}\| \leq \delta^{G}\|\Delta\bfv^k\| 
\label{convrateofvk}\eeq
\beq
\|\stacktwo{\bfw^{k+2}}{\bflambda^{k+2}} - \stacktwo{\bfy^*}{\bflambda^*/\beta}\| \leq \delta^G
\|\stacktwo{\bfw^{k}}{\bflambda^{k}} - \stacktwo{\bfy^*}{\bflambda^*/\beta}\| \label{convrateofwklamk}
\eeq
where, 
\beq\bal
	&\; \delta^{G} = \max\left( \delta\left(\|\bfM_{\bfZ}\|,c_F^*,1.0\right), \delta\left(\|\bfM_{\bfZ}\|,1.0,\alpha^{\max}(\Delta^0)\right) \right)  \\
	&\; \Delta^0 = \|\bfv^0 - \bfv^*\|.
\eal\label{defDeltaG}\eeq
\end{theorem}
\begin{proof}
At any iterate $k$ of the algorithm one of the following holds:
\begin{enumerate}[(a)]
\item $\activeset^k \subseteq \activeset^*$. \label{Ak:casea}  
Theorem~\ref{thm:localconvergence} yields a worst-case contraction factor by assuming 
$\alpha^k = 1$ as $\delta(\|\bfM_{\bfZ}\|,c_F^*,1.0) < 1$.
\item $\activeset^k \nsubseteq \activeset^*$. \label{Ak:caseb}  Note that 
$\{\|\Delta\bfv^k\|\}$ is non-increasing (refer Lemma~\ref{lemm:iternonexpansive}
\eqref{eqmvk}).  Consequently, $\Delta^k \leq \Delta^0 \;\forall\; k$ which implies that $
\alpha^{\max}(\Delta^k) \leq \alpha^{\max}(\Delta^0) \;\forall\; k$.  Thus, we can provide a uniform 
upper bound on $\alpha^k$ as $\alpha^k \leq \alpha^{\max}(\Delta^0)$.  
Combining this with Theorem~\ref{thm:globalconvergence} we obtain 
$\delta(\|\bfM_{\bfZ}\|,c_F^k,\alpha^{\max}(\Delta^k)) \leq \delta(\|\bfM_{\bfZ}\|,1.0,\alpha^{\max}(\Delta^0)) < 1$.  
\end{enumerate}
Combining the observation in the above cases and noting that $\alpha^{\max}(\|\Delta\bfv^0\|) < 1$ we have that $\delta^G$~\eqref{defDeltaG} upper bounds the contraction 
factors in all of the above cases.  Thus, the inequality in~\eqref{convrateofvk} holds.  From Lemma~\ref{lemm:iternonexpansive}\eqref{bnditersbyv}, we have that
\[
\|\stacktwo{\bfw^{k+2}}{\bflambda^{k+2}} - \stacktwo{\bfy^*}{\bflambda^*/\beta}\| \leq \|\Delta\bfv^{k+1}\|
\]
and by Lemma~\ref{lemm:iternonexpansive}\eqref{bndv} we have that,
\[
\|\Delta\bfv^k\| \leq \|\stacktwo{\bfw^{k}}{\bflambda^{k}} - \stacktwo{\bfy^*}{\bflambda^*/\beta}\|.
\]
Combining the two inequalities with~\eqref{convrateofvk} yields~\eqref{convrateofwklamk}. 
\end{proof}

The analysis shows that we expect the iterates $\{\bfv^k\}$ to have different convergence rates 
depending on whether the active set has been correctly identified or not.   
Also, note that we do not assume that $\activeset^k = \activeset^*$ holds for all $k$ 
sufficiently large.  We will explore this further in the section on numerical experiments.  

\subsection{Optimal Choice of $\beta$}\label{sec:optbeta}
Observe that $\beta$ affects the convergence rate in Theoerem~\ref{thm:qlinearconvergence} 
through $\|\bfM_{\bfZ}\|$.  Thus, we can postulate an optimum $\beta^*$ so as to minimize 
$\|{\bfM}_{\bfZ}\|$. The eigenvalues of $\bfZ^T\bfM\bfZ$ satisfy $\lambda(\bfZ^T\bfM\bfZ) = 
\lambda((\bfZ^T(\bfQ/\beta+\bfI_n)\bfZ)^{-1}) = \beta/(\beta+\lambda({\bfZ^T\bfQ\bfZ}))$. 
Thus, the optimal choice for $\beta$ is given by,
\[
\beta^{*} = \text{arg}\min\limits_{\beta > 0}\max\limits_{i}\left| \frac{\beta }{\beta + 
\lambda_i(\bfZ^T\bfQ\bfZ)} -\half\right|
\]
where we have divided $\|\bfM_{\bfZ}\|$ by $2$. 
We can rearrange the right hand side to obtain,
\beq
\beta^{*} = \text{arg}\min\limits_{\beta > 0}\max\limits_{i}\left| 
\frac{\beta/\lambda_i({\bfZ^T\bfQ\bfZ})}{\beta/\lambda_i({\bfZ^T\bfQ\bfZ}) + 1} -\half\right|. \label{optbetaform}
\eeq
Equation~\eqref{optbetaform} is identical in form to Equation~(36) of \cite{GhaTeiSha13} and 
the analysis proposed in~\cite{GhaTeiSha13} to obtain the optimal parameter can be utilized. 

 \begin{theorem}\label{thm:optstepsize} Suppose Assumptions 
\ref{ass:feas}-\ref{ass:licq} hold. Then, the optimal step-size is 
\beq
\beta^{*} = 
\sqrt{\lambda_{\text{min}}({\bfZ^T\bfQ\bfZ})\lambda_{\text{max}}({\bfZ^T\bfQ\bfZ})}.\label{optbeta}
\eeq
\end{theorem}
\begin{proof}
The proof is similar to that of Theorem 4 in~\cite{GhaTeiSha13}, and is not repeated.
\end{proof}

\subsection{Relation to Recent Literature}\label{sec:relatedlit}
In this section, we relate the convergence results to those presented in recent papers.
\subsubsection*{Ghadimi et al~\cite{GhaTeiSha13}}\label{sec:linindepcons}
Ghadimi et al~\cite{GhaTeiSha13} prove global Q-linear convergence for,  
\beq\bal
\min\limits_{y} &\; \half y^TQy + q^Ty \\
\text{s.t.} &\; Ax \leq b
\eal\label{ghadimiqp}\eeq
where $Q$ is strictly convex and $A$ is full row rank. In this setting, inequalities that have 
finite lower and upper bounds cannot be handled.  Through the introduction of slacks, 
$s \geq 0$ the constraints can be converted to $Ax+s = b, s\geq 0$.  For this formulation, it is 
easy to show that Assumption~\ref{ass:redhesspd} and~\ref{ass:licq} hold.  Hence, the results 
in this paper can be easily applied to also obtain the global Q-linear convergence result.   
On the other hand, the analysis 
in this paper can also handle equality constraints and inequalities that have finite 
lower and upper bounds. Thus our analysis strictly generalizes that in~\cite{GhaTeiSha13}.

\subsubsection*{Raghunathan \& Di Cairano~\cite{RagSte13ACC}}
In the case of QPs with no equality constraints, Assumption~\ref{ass:redhesspd} implies that $\bfQ$ is positive definite on the full space.  In this setting, 
$\bfM_{\bfZ} = 2(\bfQ/\beta+\bfI_n)^{-1} - \bfI_n$, $c_F^* = 0$. 
We obtain the global Q-linear convergence rate explicitly as 
$\half(\|\bfM_{\bfZ}\|+1)$, which is the first row in Table~\ref{tab:deltacF}. 
Thus, by~\eqref{optbeta}  
 $\beta^* = \sqrt{\lambda_{\min}(\bfQ)\lambda_{\max}(\bfQ)}$ obtaining the results in~\cite{RagSte13ACC}.

\subsubsection*{Bauschke et al~\cite{BauBelNgh14} }

Bauschke et al~\cite{BauBelNgh14} show that rate of convergence of DR for finding 
a point in the intersection of two linear subspaces is equal to the cosine of the Friedrich's 
angle.  The analysis in \S\ref{sec:onestepconvergence} does not yield a contraction 
when $\bfQ = 0$ since this violates Assumption~\ref{ass:redhesspd}.  However, since 
strict complementarity $\bfw^* + \bflambda^* \neq 0$ is not assumed, our DR iterations may not 
reduce to alternating projection between subspaces.  We can provide global Q-linear 
convergence under additional assumptions.


\subsubsection*{Boley~\cite{Bol14}}
Boley~\cite{Bol14} considers the identical QP~\eqref{qpform} as in this paper.  The author 
analyzes the convergence of the sequence $\{\bfu^k\}$ and shows that ADMM has 
$4$ different convergence regimes based on the 
eigenvalues of a certain matrix $M^{[k]}$ being $< 1$ or equal to $1$.  
Further, under 
assumptions of strict complementarity, uniqueness of primal and dual solutions 
Boley~\cite[Theorem 6.4]{Bol14} established local 
Q-linear convergence.  No analysis of the global 
behavior is provided.  For instance, in regimes (b) and (d)~\cite[\S 5.2]{Bol14} it is identified 
that matrix $M^{[k]}$ has eigenvalue $1$ and this corresponds to a change in active set.  
However, no analysis of the convergence rate is provided. 
We interpret these assumptions and results in the context of our paper.  
\begin{itemize}
\item In our notation $M^{[k]} = \bfM - \bfE^k(\bfE^k)^T$.  Under 
Assumptions~\ref{ass:redhesspd} and~\ref{ass:licq} it is easy to show that $\|M^{[k]}\| < 1$ 
once $\activeset^k \subseteq \activeset^*$.  Consequently, the regime (c) 
in~\cite[\S5.2]{Bol14} cannot occur.  The regimes (b) and (d) in~\cite[\S5.2]{Bol14} are not 
separately identified in our analysis.  Thus, only 2 convergence regimes exist 
in our analysis.  
\item Our Assumption~\ref{ass:licq} is consistent with the assumption of unique dual solution in \cite[Theorem 6.4]{Bol14}.
\item Recall that based on our definition of 
$\activeset^k$~\eqref{optactiveset}, $\activeset^k \subseteq \activeset^*$ 
for all $k$ sufficiently large.  However, 
we cannot ensure that $\activeset^k = \activeset^*$ since we do not assume strict 
complementarity.  
\item We do not require strict complementarity but we require positive definiteness of 
reduced Hessian.  Our analysis leads to global Q-linear convergence as opposed to 
local Q-linear convergence in~\cite{Bol14}.
\item We also provide a computable convergence rate while no such specification exists 
in~\cite{Bol14}.
\end{itemize}

\subsection*{Giselsson and Boyd~\cite{GisBoy14}}

Giselsson and Boyd~\cite{GisBoy14} derive Q-linear  convergence  rate  bounds  for
Douglas-Rachford splitting under strong convexity and smoothness assumptions. 
In the context of the splitting that we consider in~\eqref{qpformsplit} the results of Theorem 2 
in~\cite{GisBoy14} are not applicable since the dual function is not strongly convex.  
However, the authors do propose a heuristic for the selection of the parameter $\beta$.  
Under Assumption~\ref{ass:redhesspd} this heuristic selection coincides 
with $\beta^*$ in Theorem~\ref{thm:optstepsize}.

\subsubsection*{Liang et al~\cite{LiaFadPey15}}
This paper was published after our initial submission but we include this for 
completeness.
Liang et al~\cite{LiaFadPey15} characterize the finite active set identification and local linear 
convergence for the DR.  We interpret the main assumptions in that paper in the context of 
the QP~\eqref{qpform} as: (a) strict complementarity (eqn (3.1) 
in~\cite{LiaFadPey15}) and (b) LICQ which is required to guarantee than the angle 
between the tangent spaces is bounded away from $0$.  No assumptions on the curvature 
of the Hessian of the objective is made.  The authors show local R-linear convergence 
provided that the cosine of the Friedrich's angle between the tangent spaces is 
bounded away from 1.  Our  
analysis in \S\ref{sec:onestepconvergence} does not yield a contraction in this case.  We 
relax the requirement of strict complementarity which implies that finite active set identification 
property does not hold.  In fact based on~\eqref{optactiveset} we only have $\activeset^k 
\subseteq \activeset^*$.  
However, we require positive definiteness of the reduced Hessian to yield Q-linear convergence.

\section{Infeasible QPs}\label{sec:infeasqp}

In this section we characterize the limit of ADMM iterates when the QP in~\eqref{qpform} is 
infeasible.  The main result is that $\{\bfy^k\}$ and $\{\bfw^k\}$ converge to minimizers of the Euclidean distance between the affine subspace defined by $\bfA\bfy = b$ 
and the set $\bfYset$ and the divergence in the iterates is restricted to the multipliers along the range space of the constraints.  We assume the following for the rest of this section.
\begin{assumption}\label{ass:infeasqp}
The QP in~\eqref{qpform} is infeasible and Assumptions~\ref{ass:feas}-\ref{ass:redhesspd} hold.
\end{assumption}

The roadmap of the analysis is as follows.  \S\ref{sec:infeasMinimizer} defines the 
infeasibility minimizer for~\eqref{qpform}.    \S\ref{sec:infeasLimitADMM} 
proves the main result on the sequence to which ADMM iterates converge when QP~\eqref{qpform} is infeasible.  Finally, we discuss termination conditions 
that can be checked for detecting infeasible problems in \S\ref{sec:infeasQPTermConds}.  

\subsection{Infeasibility Minimizer}\label{sec:infeasMinimizer}

From the optimality conditions for minimizer of infeasibility~\eqref{infeasmin} it is clear that 
the point $\stacktwo{\bfy^\circ}{\bfw^\circ}$ is only unique along the range of $\bfR$. There may exist multiple solutions when a direction 
along the range of $\bfZ$ is also a direction from $\bfw^\circ$ leading into the convex set $\bfYset$. In other words, 
$\stacktwo{\bfy^\circ+\bfZ\bfy_{\bfZ}}{\bfw^\circ+\bfZ\bfy_{\bfZ}}$ are also minimizers of the Euclidean distance between the hyperplane $\bfA\bfy = \bfb$ 
and the convex set $\bfYset$.  In the following we refine the notion of infeasibility minimizer while 
accounting for the effect of the objective function.  This is essential since in the 
ADMM iterations the update step for $\bfy$ does account for the objective function.  We prove the existence of $\bfy^{\bfQ}, \bflambda^{\bfQ}$ 
which is used subsequently in Theorem~\ref{thm:fixpointsinfeas} to show that the sequence 
$\{\stackthr{\bfy^\circ+\bfy^{\bfQ}}{\bfw^\circ+\bfw^{\bfQ}}{\frac{1}{\beta}(\gamma^k\bflambda^\circ+\bflambda^{\bfQ})}\}$ where $\gamma^k - \gamma^{k-1} = 1$ 
satisfies the ADMM iterations in~\eqref{admmiter_ext}.  

\begin{lemma}\label{lemm:infeasptshift}
Suppose Assumption~\ref{ass:infeasqp} holds. Then, there exists 
$\bfy^{\bfQ} \in \text{range}(\bfZ)$, $\bflambda^{\bfQ} \in \R^n$, with 
$\bfy^{\bfQ}$, $\bfZ^T\bflambda^{\bfQ}$ unique, such that
\beq\bal
\bfZ^T\bfQ(\bfy^\circ+\bfy^{\bfQ}) + \bfZ^T\bfq - \bfZ^T\bflambda^{\bfQ} &= 0 \\
\bflambda^{\bfQ} \perp (\bfw^\circ+\bfy^{\bfQ}) &\in \bfYset.
\eal\label{defineshift}\eeq
Furthermore,  $(\bflambda^{\bfQ} + \gamma\bflambda^\circ)$ $\forall\; \gamma \geq 0$ is also a solution to~\eqref{defineshift}.
\end{lemma}
\begin{proof}
Since $\bfy^Q \in\text{range}(\bfZ)$, let $\bfy^{\bfQ} = \bfZ\bfy^{\bfQ}_{\bfZ}$ for some 
$\bfy^{\bfQ}_{\bfZ} \in \R^{n-m}$. Substituting this in~\eqref{defineshift} we obtain,
\[\bal
\bfZ^T\bfQ\bfZ\bfy^{\bfQ}_{\bfZ} + \bfZ^T(\bfq+\bfQ\bfy^\circ) - \bfZ^T\bflambda^{\bfQ} &= 0 \\
\bflambda^{\bfQ} \perp (\bfw^\circ+\bfZ\bfy^{\bfQ}_{\bfZ}) &\in \bfYset.
\eal\]
The above are the optimality conditions for,
\beq\bal
\min\limits_{\bfy^{\bfQ}_{\bfZ}} &\; \half(\bfy^{\bfQ}_{\bfZ})^T(\bfZ^T\bfQ\bfZ)\bfy^{\bfQ}_{\bfZ} + 
(\bfZ^T\bfq+\bfZ^T\bfQ\bfy^\circ)^T\bfy^{\bfQ}_{\bfZ} \\ 
\text{s.t.} &\; \bfw^\circ+\bfZ\bfy^{\bfQ}_{\bfZ} \in \bfYset.
\eal\label{qpform4shift}\eeq
The strict convexity of the QP~\eqref{qpform4shift} follows from 
Assumption~\ref{ass:redhesspd} and this guarantees uniqueness of $\bfy^{\bfQ}_{\bfZ}$, 
if one exists. Further, weak Slater's 
condition~\cite{BoydVandenbergheBook} holds for the 
QP~\eqref{qpform4shift} since the constraints in $\bfYset$ are affine and 
$\bfy^{\bfQ}_{\bfZ}=0$ is a feasible point.  The satisfaction of convexity and weak Slater's condition 
by QP~\eqref{qpform4shift} implies that strong duality holds for~\eqref{qpform4shift} 
and the claim on existence of $\bfy^{\bfQ}_{\bfZ}, \bflambda^{\bfQ}$ holds. The uniqueness of 
$\bfy^{\bfQ}$ follows from the uniqueness of $\bfy^{\bfQ}_{\bfZ}$ and the full column rank of $\bfZ$.  
The uniqueness of $\bfZ^T\bflambda^{\bfQ}$ follows from the first equation of~\eqref{defineshift}
and the uniqueness of $\bfy^{\bfQ}$.

To prove the remaining claim, consider the choice of 
$(\bflambda^{\bfQ} + \gamma\bflambda^\circ)$ as a solution 
to~\eqref{defineshift}. Satisfaction of the first equation in~\eqref{defineshift} 
follows from $\bflambda^{\circ} \in \text{range}(\bfR)$ by~\eqref{diffywcirc} and~\eqref{RZorthogonal}. As for the variational inequality in~\eqref{defineshift}, 
\[\bal
&\; (\bflambda^{\bfQ} + \gamma\bflambda^\circ)^T(\bfw' - (\bfw^\circ+\bfy^{\bfQ})) \\
=&\; \underbrace{(\bflambda^{\bfQ})^T(\bfw' - (\bfw^\circ+\bfy^{\bfQ}))}_{\geq\;0} +  \underbrace{ \gamma(\bflambda^\circ)^T(\bfw' - \bfw^\circ)}_{\geq\;0}  
-\underbrace{ \gamma(\bflambda^\circ)^T\bfy^{\bfQ} }_{=\;0} \geq 0 \;\forall\;\bfw' \in \bfYset
\eal\]
where the first term is non-negative by the variational 
inequality in~\eqref{defineshift}, the second term is non-negative by the  
variational inequality in~\eqref{diffywcirc} and the last term vanishes since 
$\bflambda^\circ \in \text{range}(\bfR)$ and $\bfy^{\bfQ} \in \text{range}(\bfZ)$. Thus, 
$(\bflambda^{\bfQ} + \gamma\bflambda^\circ)$ satisfies the variational inequality 
in~\eqref{defineshift} for all $\gamma \geq 0$.
\end{proof}

\subsection{Limit Sequence for ADMM}\label{sec:infeasLimitADMM}

The following result characterizes the limit behavior of ADMM iterates for infeasible instances of QP~\eqref{qpform} in terms of the sequence $\{\bfv^k\}$. 
\begin{lemma}\label{lemm:limitofvk}
Suppose Assumption \ref{ass:infeasqp} holds. Then,
\beq
\lim_{k \rightarrow \infty}\|\bfv^{k+1}-\bfv^k\|  = \omega \neq 0. \label{eqmvkinfeas}
\eeq 
Further, the ADMM iterates satisfy, 
\beq\bal
&\; \lim_{k \rightarrow \infty} \bfy^k = \bar{\bfy},\; \lim_{k \rightarrow \infty}\bfw^k = \bar{\bfw},\; \lim_{k \rightarrow \infty}\bfZ^T\bflambda^k = \bar{\bflambda}_{\bfZ}, \\ 
&\; \lim_{k \rightarrow \infty} \|\bfR^T(\bflambda^{k+1}-\bflambda^k)\| = \omega,
\text{ and } \bar{\bfw}-\bar{\bfy} \in \text{range}(\bfR).
\eal \label{limitofinfeasadmmiters}
\eeq
\end{lemma}
\begin{proof}
From Lemma~\ref{lemm:iternonexpansive}\eqref{eqmvk}, we have that $\{\|\bfv^{k}-\bfv^{k-1}\|\}$ is a bounded, 
non-increasing sequence of nonnegative real numbers.  Hence, there exists a limit for the above sequence which we denote, $\lim_{k \rightarrow \infty}\|\bfv^k-\bfv^{k-1}\| = \omega$.  Since QP~\eqref{qpform} is infeasible by Assumption~\ref{ass:infeasqp}, we must necessarily have that $\omega > 0$. Consider the the following redefinition of the quantities
\begin{subequations}
\beq
\Delta\bfv^k = \bfv^k - \bfv^{k-1},\;\;\; \Delta\bfu^k = \bfu^k - \bfu^{k-1}. \label{defdeltaqts}
\eeq
Taking the differences over successive DR iterations~\eqref{drform} and substituting 
$\bfM = \bfZ\bfM_{\bfZ}\bfZ^T - \bfR\bfR^T$ from~\eqref{defMNMZ} we have that,
\[\bal
\Delta\bfv^{k+1} =&\; \half(\bfZ\bfM_{\bfZ}\bfZ^T - \bfR\bfR^T)\Delta\bfu^k + \half\Delta\bfv^k \\
\implies \|\Delta\bfv^{k+1}\| \leq&\; \half\|(\bfZ\bfM_{\bfZ}\bfZ^T - \bfR\bfR^T)\Delta\bfu^k\| + 
\half\|\Delta\bfv^k\| \\
\implies \|\Delta\bfv^{k+1}\| \leq&\; \half\sqrt{\|\bfM_{\bfZ}\|^2\|\bfZ^T\Delta\bfu^k\|^2 + 
\|\bfR^T\Delta\bfu^k\|^2} + \half\|\Delta\bfv^k\| 		
\eal\]
where the first inequality is obtained from triangle inequality and the second is obtained from 
Cauchy-Schwartz and~\eqref{RZorthogonal}. Since $\|\bfM_{\bfZ}\| < 1$,  it follows that $\|\Delta\bfv^{k+1}\| < \|\Delta\bfv^k\|$ for all $k$ such that $\|\bfZ^T\Delta\bfu^k\| \neq 0$.  
Since $\{\|\Delta\bfv^k\|\}$ converges to $\omega > 0$ it must be true that $\{\bfZ^T\Delta\bfu^k\} \rightarrow 0$.  From the update step for $\bfy$~\eqref{admmiter_ext_y} 
and~\eqref{relateadmmdr},
\beq\bal
\bfy^{k+1}-\bfy^k =&\; \bfM(\bfw^k+\bflambda^k) - \bfM(\bfw^{k-1}-\bflambda^{k-1})  \\
=&\; \bfM(\bfu^{k-1} - \bfu^{k-2}) = \bfZ(\bfZ^T\bfQ\bfZ/\beta+\bfI_{n-m})^{-1}\bfZ^T\Delta\bfu^{k-1}
\eal\eeq
where the final equality is obtained from substitution of $\bfM$.  Hence, $\{\bfy^{k+1}-\bfy^k\} \rightarrow 0$ from the convergence of $\{\bfZ^T\Delta\bfu^k\} \rightarrow 0$.  From the convergence of $\{\bfy^k\}$ and~\eqref{relateadmmdr},
\beq
\lim_{k \rightarrow \infty} \|\bfv^{k+1}-\bfv^k\| = \lim_{k \rightarrow \infty} \|(\bfy^{k+2}-\bflambda^{k+1})-(\bfy^{k+1}-\bflambda^k)\| = \lim_{k \rightarrow \infty} \|\bflambda^{k+1}-\bflambda^k\| = \omega . \label{limofvislimoflam}
\eeq
\end{subequations}
To show the convergence of $\bfw^k$ note by Lemma~\ref{lemm:iternonexpansive}\eqref{bnditersbyv} that,
\[
\|\stacktwo{\bfw^{k+1}}{{\bflambda}^{k+1}}-\stacktwo{\bfw^k}{{\bflambda}^k}\| \leq  \|\bfv^{k}-\bfv^{k-1}\| \implies \lim_{k \rightarrow \infty} \|\bfw^{k+1}-\bfw^k\| = 0 
\]
where the implication follows by taking limits on both sides and using~\eqref{limofvislimoflam}. Further,
\[
\left.\bal	&\; \lim_{k \rightarrow \infty}(\bfw^{k+1}-\bfw^k)  = 0 \\
		&\; \lim_{k \rightarrow \infty}\bfZ^T\Delta\bfu^{k+1}  = 0 \eal\right\} 
		\overset{\eqref{relateadmmdr}}{\implies}
		\lim_{k \rightarrow \infty}\bfZ^T(\bflambda^{k+1}-\bflambda^k) = 0. 
\]
Combining this with~\eqref{limofvislimoflam} we obtain the said results on $\bflambda$ in~\eqref{limitofinfeasadmmiters}.  From the update step~\eqref{admmiter_ext_lam} it follows that,
\[
\lim_{k \rightarrow \infty} (\bfw^{k+1}-\bfy^{k+1}) = \lim_{k \rightarrow \infty} (\bflambda^{k+1}-\bflambda^{k}) \in \text{range}(\bfR) \implies \bar{\bfw}-\bar{\bfy} \in \text{range}(\bfR)
\]
where the first inclusion follows from $\bfZ^T(\bflambda^{k+1}-\bflambda^k) \rightarrow 0$ and this proves the remaining claim in~\eqref{limitofinfeasadmmiters}.
\end{proof}

The next lemma establishes some properties of the ADMM iterate sequence.
\begin{lemma}\label{lemm:infeasconds}
Suppose Assumption~\ref{ass:infeasqp} holds.  Then the iterates $\{\stackthr{\bfy^k}{\bfw^k}{\bflambda^k}\}$ 
generated by the ADMM algorithm in~\eqref{admmiter_ext} satisfy,
\beq
\left\{\frac{ ({\bflambda}^k) ^T({\bfw}^k-{\bfy}^k)}{\|{\bflambda}^k\|\|{\bfw}^k-{\bfy}^k\|}\right\} \rightarrow 1.\label{infeasangleconds}
\eeq
\end{lemma}
\begin{proof}
To show~\eqref{infeasangleconds}, suppose the following holds,
\begin{subequations}
\beq
\left\{\frac{ ({\bflambda}^k) ^T(\bar{\bfw}-\bar{\bfy})}{\|{\bflambda}^k\|\|\bar{\bfw}-\bar{\bfy}\|}\right\} \rightarrow 1 \label{infeasangleconds1}
\eeq
where $\bar{w},\bar{y}$ are as defined in~\eqref{limitofinfeasadmmiters}. 
Consider the following decomposition,
\beq
\bfw^k - \bfy^k = \theta^k(\bar{\bfw}-\bar{\bfy}) + \bfnu^k 
\text{ where, } (\bar{\bfw}-\bar{\bfy})^T\bfnu^k = 0.
\label{barwydecomp}
\eeq 
 By~\eqref{limitofinfeasadmmiters}, $\{\theta^k\} \rightarrow 1$, $\{\bfnu^k\} \rightarrow 0$. Using~\eqref{barwydecomp} we have,
\beq\bal
\frac{ ({\bflambda}^k) ^T({\bfw}^k-{\bfy}^k)}{\|{\bflambda}^k\|\|{\bfw}^k-{\bfy}^k\|} =&\; 
\frac{ ({\bflambda}^k) ^T(\theta^k(\bar{\bfw}-\bar{\bfy}) + \bfnu^k)}{\|{\bflambda}^k\|\|\theta^k(\bar{\bfw}-\bar{\bfy}) + \bfnu^k\|} 
\geq \frac{ ({\bflambda}^k) ^T(\theta^k(\bar{\bfw}-\bar{\bfy}) + \bfnu^k)}{\|{\bflambda}^k\|(\|\theta^k(\bar{\bfw}-\bar{\bfy})\| + \|\bfnu^k\|)} \\
=&\;\frac{ ({\bflambda}^k) ^T(\theta^k(\bar{\bfw}-\bar{\bfy}) + \bfnu^k)}{\|{\bflambda}^k\|\|\theta^k(\bar{\bfw}-\bar{\bfy})\| }\left(1 + \frac{\|\bfnu^k\|}{\|\theta^k(\bar{\bfw}-\bar{\bfy})\|}\right)^{-1}
\eal\label{proveinfeasangleconds}\eeq
where the first inequality follows from applying triangle inequality to $\|\theta^k(\bar{\bfw}-\bar{\bfy}) + \bfnu^k\|$ and the second equality 
by rearrangement of terms. 
As $k \rightarrow \infty$ the last term in~\eqref{proveinfeasangleconds} approaches~\eqref{infeasangleconds1} and hence,~\eqref{infeasangleconds} holds if~\eqref{infeasangleconds1} is true. 
To show~\eqref{infeasangleconds1} consider 
\beq\bal
&\; \bflambda^{k+l} = \bflambda^k + \sum\limits_{j=1}^{l}(\bfw^{k+j}-\bfy^{k+j}) = \bflambda^k + \bar{\theta}^{k+l}(\bar{\bfw}-\bar{\bfy}) + \bar{\bfnu}^{k+l} \\
\text{with} &\; \bar{\theta}^{k+l} = \sum_{j=1}^l\theta^{k+j}, \bar{\bfnu}^{k+l} = \sum_{j=1}^l\bfnu^{k+j}
\eal\label{decomplam}
\eeq
which is obtained by summing~\eqref{admmiter_ext_lam} over iterations 
$k,\ldots,k+l-1$ and substituting~\eqref{barwydecomp}.  Substituting~\eqref{decomplam} in~\eqref{infeasangleconds1} we obtain,
\[\bal
\frac{ ({\bflambda}^{k+l}) ^T(\bar{\bfw}-\bar{\bfy})}{\|{\bflambda}^{k+l}\|\|\bar{\bfw}-\bar{\bfy}\|} 
=&\; \frac{\left(\bflambda^k + \bar{\theta}^{k+l}(\bar{\bfw}-\bar{\bfy}) \right)^T(\bar{\bfw}-\bar{\bfy})}{\|\bflambda^k + \bar{\theta}^{k+l}(\bar{\bfw}-\bar{\bfy}) + \bar{\bfnu}^{k+l}\|\|\bar{\bfw}-\bar{\bfy}\|} \\
\geq&\; \frac{\left(\bflambda^k + \bar{\theta}^{k+l}(\bar{\bfw}-\bar{\bfy}) \right)^T(\bar{\bfw}-\bar{\bfy})}{(\|\bar{\theta}^{k+l}(\bar{\bfw}-\bar{\bfy})\| + \|\bflambda^k + \bar{\bfnu}^{k+l}\|)\|\bar{\bfw}-\bar{\bfy}\|} \\
=&\;\left(1 + \frac{(\bflambda^k)^T(\bar{\bfw}-\bar{\bfy})}{\bar{\theta}^{k+l}\|\bar{\bfw}-\bar{\bfy}\|^2}\right)\left(1 + \frac{\|\bflambda^k + \bar{\bfnu}^{k+l}\|}{\bar{\theta}^{k+l}\|\bar{\bfw}-\bar{\bfy}\|} \right)^{-1} \\
\geq&\; \left(1 - \frac{\|\bflambda^k\|}{\bar{\theta}^{k+l}\|\bar{\bfw}-\bar{\bfy}\|}\right)\left(1 + \frac{\|\bflambda^k + \bar{\bfnu}^{k+l}\|}{\bar{\theta}^{k+l}\|\bar{\bfw}-\bar{\bfy}\|} \right)^{-1}
\eal\]
where the first equality follows from $(\bfnu^k)^T(\bar{\bfw}-\bar{\bfy}) =0$ and the first inequality from applying the triangle inequality to  
$\|\bflambda^k + \bar{\theta}^{k+l}(\bar{\bfw}-\bar{\bfy}) + \bar{\bfnu}^{k+l}\|$, the second equality follows simply by rearranging and the final inequality from applying 
Cauchy-Schwarz inequality.  From~\eqref{decomplam},
\[
\bfZ^T\bflambda^{k+l}= \bfZ^T\bflambda^k + \bfZ^T\bar{\bfnu}^{k+l} \implies \{\bfZ^T\bflambda^k + \bfZ^T\bar{\bfnu}^{k+l}\} \rightarrow \bar{\bflambda}_{\bfZ}
\]
where the first equality follows from $(\bar{\bfw}-\bar{\bfy}) \in \text{range}(\bfR)$ and the implication follows from~\eqref{limitofinfeasadmmiters}. Further, $\bfnu^k \in \text{range}(\bfZ)$ since $(\bar{\bfw} - \bar{\bfy})^T\bfnu^k = 0$.  Hence,  
\[
\lim_{l \rightarrow \infty} \|\bflambda^k + \bar{\bfnu}^{k+l}\|^2 = \|\bfR^T\bflambda^k\|^2 + \lim_{l \rightarrow \infty} \|\bfZ^T\bflambda^k + \bfZ^T\bar{\bfnu}^{k+l}\|^2 = \|\bfR^T\bflambda^k\|^2 +
\|\bar{\bflambda}_{\bfZ}\|^2
\] 
is bounded.  On the other hand, $\{\bar{\theta}^{k+l}\}\rightarrow \infty$ as $l\rightarrow\infty$ 
since $\|\bflambda^{k+1}\| \rightarrow \infty$ by~\eqref{limitofinfeasadmmiters} and hence,  in the limit $l \rightarrow \infty$ we obtain~\eqref{infeasangleconds1}.   This completes the proof.
\end{subequations}
\end{proof}

Using Lemmas~\ref{lemm:infeasptshift} and~\ref{lemm:infeasconds} we can state the 
limiting behavior of the ADMM iterations~\eqref{admmiter_ext} when the QP~\eqref{qpform} 
is infeasible.

\begin{theorem}\label{thm:fixpointsinfeas}
Suppose Assumptions~\ref{ass:infeasqp} holds. Then, the following statements are true.
\begin{enumerate}[(i)]
\item\label{thminfeasqp} If QP~\eqref{qpform} is infeasible then,  
$\{\stackthr{\bfy^\circ+\bfy^{\bfQ}}{\bfw^\circ+\bfy^{\bfQ}}{\widehat{\bflambda}^{k}}\}$ 
is a sequence satisfying~\eqref{admmiter_ext} for $k \geq k'$ sufficiently large with, 
$\bfy^{\bfQ}, \bflambda^{\bfQ}$ as defined in~\eqref{defineshift} and, 
\beq
\widehat{\bflambda}^{k} = \frac{1}{\beta}(\bflambda^{\bfQ}+(k-\gamma_1)\bflambda^\circ), 
 \gamma_1 \leq k'.
\label{infeaslam}\eeq
\item\label{thmlamiterunbnd} If the ADMM algorithm~\eqref{admmiter_ext} generates 
$\{\stackthr{{\bfy}^k}{{\bfw}^k}{{\bflambda}^k}\}$ satisfying~\eqref{limitofinfeasadmmiters}
then, the QP~\eqref{qpform} is infeasible. Further, $\bar{\bfy} = \bfy^\circ+\bfy^{\bfQ}$, 
$\bar{\bfw} = \bfw^\circ+\bfw^{\bfQ}$ and $\bflambda^k$ satisifes~\eqref{infeaslam}.\label{thm:seqimpliesinfeasQP}
\end{enumerate}
\end{theorem}
\begin{proof} 
\begin{subequations}
Consider the claim in~\eqref{thminfeasqp}. For proving that \eqref{admmiter_ext_y} holds, we 
need to show that,
\beq
\bfy^\circ+\bfy^{\bfQ}-\bfM(\bfw^\circ+\bfy^{\bfQ}+\widehat{\bflambda}^k-\tilde{\bfq}) - \bfN\bfb = 0.\label{admmiter_ext_y_condn}
\eeq
Multiplying the left hand side of~\eqref{admmiter_ext_y_condn} by $\bfR^T$, using $\bfR^T\bfM = 0$, $\bfR^T\bfy^{\bfQ} = 0$ and simplifying,
\beq
\bfR^T\bfy^\circ - (\bfA\bfR)^{-1}\bfb = (\bfA\bfR)^{-1}(\bfA\bfR\bfR^T\bfy^\circ - \bfb) = 0 \label{admmiter_ext_y_condnR}
\eeq
where the last equality follows from~\eqref{diffywcirc}. 
Multiplying the left hand side of~\eqref{admmiter_ext_y_condn} by $\bfZ^T$, and substituting for 
$\bfZ^T\bfM = \widehat{\bfM}\bfZ^T$ 
and  $\bfZ^T\bfN\bfb = -(\widehat{\bfM}\bfZ^T\bfQ/\beta)\bfR\bfR^T(\bfy^\circ+\bfy^{\bfQ})$ 
where $\widehat{\bfM} = (\bfZ^T\bfQ\bfZ/\beta+\bfI_{n-m})^{-1}$ 
we obtain,
\beq\bal
&\;\bfZ^T(\bfy^\circ + \bfy^{\bfQ}) - \widehat{\bfM}\bfZ^T(\bfw^\circ + \bfy^{\bfQ} + \widehat{\bflambda}^k - \tilde{\bfq}) + \widehat{\bfM}\bfZ^T(\bfQ/\beta)\bfR\bfR^T(\bfy^\circ+\bfy^{\bfQ}) \\
=&\;\widehat{\bfM}\left(\bal  (\bfZ^T\bfQ\bfZ/\beta+\bfI_{n-m})\bfZ^T(\bfy^\circ+\bfy^{\bfQ}) \\
				- \bfZ^T(\bfw^\circ+\bfy^{\bfQ} +\widehat{\bflambda}^k- \tilde{\bfq}) + \bfZ^T(\bfQ/\beta)\bfR\bfR^T(\bfy^\circ+\bfy^{\bfQ}) \eal\right) \\
=&\;\widehat{\bfM}
				\bfZ^T(\bfQ/\beta)(\bfy^\circ+\bfy^{\bfQ}) + \bfZ^T(\bfy^\circ+\bfy^{\bfQ}) 
				- \bfZ^T(\bfw^\circ + \bfy^{\bfQ} + {\bflambda}^{\bfQ} - \tilde{\bfq}) \\
=&\;(\widehat{\bfM}/\beta)\left(\bfZ^T\bfQ(\bfy^\circ+\bfy^{\bfQ}) + \bfZ^T\bfq - \bfZ^T\bflambda^{\bfQ} \right) = 0
\eal\label{admmiter_ext_y_condnZ}\eeq
where the first equality follows simply by removing $\widehat{\bfM}$ as the common multiplicative factor, the second equality 
follows from~\eqref{RZfullspace}, the third equality from~\eqref{diffywcirc},~\eqref{infeaslam}, and the final equality from~\eqref{defineshift}.
%
Combining~\eqref{admmiter_ext_y_condnR} and ~\eqref{admmiter_ext_y_condnZ}
 shows that the sequence satisfies~\eqref{admmiter_ext_y_condn}. 
To prove that~\eqref{admmiter_ext_w} holds consider for any $\bfw' \in \bfYset$,
\beq\bal
&\;\left(\bfw^\circ+\bfy^{\bfQ} - \bfy^\circ-\bfy^{\bfQ} + \widehat{\bflambda}^{k}\right)^T\left(\bfw'-
\bfw^\circ-\bfy^{\bfQ}\right) \\
=&\; \left(\bfw^\circ - \bfy^\circ + \widehat{\bflambda}^{k}\right)^T\left(\bfw'-\bfw^\circ-\bfy^{\bfQ}\right) \\
=&\;\frac{1}{\beta}\left(\bflambda^{\bfQ}+(k-\gamma_1+1)\bflambda^\circ\right)^T
(\bfw'-\bfw^\circ-\bfy^{\bfQ}) \\
=&\; \frac{1}{\beta}\left(\bflambda^{\bfQ}+(k-\gamma_1+1)\bflambda^\circ\right)^T
\left(\bfw'-\bfw^\circ-\bfy^{\bfQ}\right) 
\geq\; 0 \\
\eal\eeq
where the second equality follows from~\eqref{diffywcirc} and~\eqref{infeaslam}, 
and the inequality follows from Lemma~\ref{lemm:infeasptshift} by noting that $\gamma = 
(k-\gamma_1+1) \geq 0$. Thus,
$\bfw^\circ+\bfw^{\bfQ} = \proj_{\bfYset}\left(\bfy^\circ+\bfy^{\bfQ}-\widehat{\bflambda}^k\right)$ 
holds and the sequence in the claim satisfies~\eqref{admmiter_ext_w}.
Finally, the definition of $\widehat{\bflambda}^{k}$ in~\eqref{infeaslam} implies 
that~\eqref{admmiter_ext_lam} holds, and 
thus~\eqref{thminfeasqp} is proved.

Consider the claim in part~\eqref{thmlamiterunbnd}. From~\eqref{infeasangleconds} we have that for any $\epsilon > 0$ there exists $k_\epsilon$ such that for all 
$k \geq k_\epsilon$,
\beq\bal
\frac{(\bflambda^k)^T(\bfw^k-\bfy^k)}{\|\bfw^k-\bfy^k\|^2} \geq (1 - \epsilon)\frac{\|\bflambda^k\|}{\|\bfw^k-\bfy^k\|}.
\eal\label{54e}\eeq
Consider the following decomposition,
\beq
\bflambda^k = \theta^k(\bfw^k-\bfy^k) + \bfmu^k 
\text{ with } (\bfw^k-\bfy^k)^T\bfmu^k = 0.  
\eeq
Then, from~\eqref{54e} we have that,
\beqar
\theta^k = \frac{(\bflambda^k)^T(\bfw^k-\bfy^k)}{\|\bfw^k-\bfy^k\|^2} \geq (1 - \epsilon)\frac{\|\bflambda^k\|}{\|\bfw^k-\bfy^k\|}, \label{defalphak}\\
\|\bfmu^k\| \leq \sqrt{1 - (1-\epsilon)^2} \|\bflambda^k\| \label{defmuk}.
\eeqar
Then for all $\bfw' \in \bfYset$ we have that,
\beq\bal
(\bfw^k-\bfy^k)^T(\bfw' - \bfw^k) 
=&\; \frac{1}{\theta^k}\underbrace{(\bflambda^{k})^T(\bfw' - \bfw^k)}_{\geq 0} - \frac{1}{\theta^k}(\bfmu^k)^T(\bfw' - \bfw^k) \\
\geq&\; -\frac{\sqrt{1 - (1-\epsilon)^2}}{1 - \epsilon}\|\bfw^k-\bfy^k\|\|\bfw'-\bfw^k\| 
\eal\eeq
where the inequality follows from Lemma~\ref{lemm:varineqholds}, the Cauchy-Schwarz inequality and the 
substitution of~\eqref{defalphak} and~\eqref{defmuk}.
Hence,
\beq\bal
&\;\lim\limits_{k \rightarrow \infty}\frac{(\bfw^k-\bfy^k)^T(\bfw' - \bfw^k) }{\|\bfw^k-\bfy^k\|\|\bfw'-\bfw^k\|} \geq 0  \;\forall\; \bfw' \in \bfYset\\
\implies&\;  \frac{(\bar{\bfw}-\bar{\bfy})^T(\bfw' - \bar{\bfw}) }{\|\bar{\bfw}-\bar{\bfy}\|\|\bfw'-\bar{\bfw}\|} \geq 0 \;\forall\; \bfw' \in \bfYset.
\eal\eeq
and $(\bar{\bfw}-\bar{\bfy}) \perp \bar{\bfw} \in \bfYset$. Since $\bfA\bar{\bfy} = \bfb$, $\bar{\bfw} \in \bfYset$ we have that 
$(\bar{\bfy}, \bar{\bfw})$ satisfies~\eqref{diffywcirc} 
and hence, the QP~\eqref{qpform} is infeasible.  
From uniqueness of the range space component in~\eqref{infeasmin},  
$\bfR^T\bar{\bfy} = \bfR^T\bfy^\circ$, $\bfR^T\bar{\bfw} = \bfR^T\bfw^\circ$ and also $\bfZ^T\bar{\bfw} = \bfZ^T\bar{\bfy}$. From the update steps in the ADMM~\eqref{admmiter_ext} we have that,
\beq\bal
\bfZ^T\left(\bfQ\left(\bfy^\circ+\bfZ\bfZ^T(\bar{\bfy}-\bfy^\circ)\right) + \bfq - \beta\bflambda^k\right) = 0, \\
\bflambda^k \perp \bfw^\circ+\bfZ\bfZ^T(\bar{\bfw}-\bfw^\circ) \in \bfYset,
\eal\label{iterlimconds}\eeq
for all $k$ sufficiently large, where first equation follows by replacing $\bfy^{\bfQ}$, 
$\bflambda^{\bfQ}$ by $\bfZ\bfZ^T(\bar{\bfy}-\bfy^\circ)$, $\beta\bflambda^k$,  respectively, 
in~\eqref{admmiter_ext_y_condnZ}, and the second condition follows from 
Lemma~\ref{lemm:varineqholds}. 
The conditions in~\eqref{iterlimconds} are precisely those in~\eqref{defineshift} and hence, Lemma~\ref{lemm:infeasptshift} applies to yield that $\bfZ\bfZ^T(\bar{\bfy}-\bfy^\circ) = \bfZ\bfZ^T(\bar{\bfw}-\bfw^\circ) = \bfy^{\bfQ}$, 
$\bfZ^T\bflambda^k = \bfZ^T\bflambda^{\bfQ}$ . 
Thus, $\bar{\bfy}= \bfy^\circ+\bfy^{\bfQ}$, $\bar{\bfw} = \bfw^\circ+\bfy^{\bfQ}$,  
$\bflambda^k$ satisfies~\eqref{infeaslam} and the claim holds.
\end{subequations}
\end{proof}

\subsection{Termination Conditions}\label{sec:infeasQPTermConds}

\begin{subequations}\label{termconds}
The termination condition in ADMM for determining an $\epsilon_o$-optimal solution is~\cite{Boyadmm},
\beq
\max(\beta\|\bfw^{k}-\bfw^{k-1}\|,\|\bflambda^{k}-\bflambda^{k-1}\|) \leq \epsilon_o.
\label{opttermconds}
\eeq
In the case of infeasible QPs, Theorem~\ref{thm:fixpointsinfeas} shows that the multipliers do not converge in the limit and increase in norm at every iteration by 
$\|\bflambda^\circ\|/\beta$.  Further, the multipliers in the limit are aligned along $\bfw^\circ-\bfy^\circ$ according to~\eqref{infeasangleconds}.  Hence, a strict termination condition 
is to monitor for the satisfaction of the conditions in~\eqref{limitofinfeasadmmiters} and~\eqref{infeasangleconds}.  A more practical approach is to consider the following set of conditions:
\beqar
\max(\beta\|\bfw^{k}-\bfw^{k-1}\|,\|\bflambda^{k}-\bflambda^{k-1}\|) > \epsilon_o. \label{termcond:err_opt}
\eeqar
\beqar
\frac{\max(\|\bfy^{k}-\bfy^{k-1}\|,\beta\|\bfw^{k}-\bfw^{k-1}\|)}{\max(\beta\|\bfw^{k}-\bfw^{k-1}\|,\|\bflambda^{k}-\bflambda^{k-1}\|)}
 \leq \epsilon_{r} \label{termcond:err_ratio}
\eeqar
\beqar
\frac{(\bflambda^k)^T(\bfw^k - \bfy^k)}{\|\bflambda^k\|\|\bfw^k-\bfy^k\|} \geq 1 - \epsilon_{a} \label{termcond:angle}
\eeqar
\beqar
\bflambda^k \circ (\bfw^k-\bfy^k) \geq 0 \text{ or } \frac{\|\Delta\bfv^{k}-\Delta\bfv^{k-1}\|}{\|\bfv^{k}\|} \leq \epsilon_{v} \label{termcond:align_blowup}
\eeqar
\end{subequations}
where, $0\leq \epsilon_o, \epsilon_r, \epsilon_a, \epsilon_v  \ll 1$, $\circ$ represents the 
componentwise multiplication (Hadamard product) and $\Delta\bfv^k = \bfv^k - \bfv^{k-1}$. 
The left hand side~\eqref{termcond:err_opt} is the error criterion used for termination in feasible QPs~\cite{Boyadmm}.  Condition~\eqref{termcond:err_opt}  
requires that the optimality conditions are not satisfied to a tolerance of $\epsilon_o$, while~\eqref{termcond:err_ratio} requires that the change in $\bfy, \bfw$ iterates 
to be much smaller than the change in the $\bfw, \bflambda$ iterates.  In the case of a feasible QP all the iterates 
converge and nothing specific can be said about this ratio.  However, as shown in 
Theorem~\ref{thm:fixpointsinfeas} the multiplier iterates change by a constant vector in the case 
of an infeasible QP.  Hence, we expect the ratio in~\eqref{termcond:err_ratio} to be small in the 
infeasible case 
while~\eqref{termcond:err_opt} is large.  The condition~\eqref{termcond:angle} checks for the satisfaction of~\eqref{infeasangleconds} to a tolerance of $\epsilon_a$.  The first condition in~\eqref{termcond:align_blowup}  
checks that each component of $\bflambda^k$ and $\bfw^k-\bfy^k$ have the same sign.  In a 
sense, this is a stricter requirement of the angle 
condition~\eqref{termcond:angle}.  In our numerical experiments we have observed that the 
satisfaction of this condition can be quite slow to converge when the iterates are 
far from a solution.  In such instances, we have also observed that, the quantity $\|\bfv^k\|$ has 
actually diverged to a large value.  To remedy this we also monitor the ratio of 
$\|\Delta\bfv^k-\Delta\bfv^{k-1}\|$ (which converges to $0$, refer Lemma~\ref{lemm:limitofvk}) 
to $\|\bfv^k\|$ ($\|\bfv^k\| \rightarrow \infty$).  We recommend following
parameter setting: $\epsilon_o = 10^{-6}, \epsilon_r = 10^{-3}, \epsilon_a = 10^{-3}, 
\epsilon_v = 10^{-4}$.  While these values have worked well on a large number of problems, 
these constants might have to be modified depending on the conditioning of the problem.

\section{Numerical Experiments}\label{sec:numerical}
In this section, we describe numerical experiments that validate the theoretical results 
obtained in the earlier sections.  \S\ref{sec:tightFeasqp} probes the tightness of the worst-
case convergence rates by varying the choice of initial iterates.  The effect of 
scaling of variables in the problem is explored in \S\ref{sec:scalingFeasqp}.  This particular 
scaling is important since this does not affect the computational complexity of computing the 
projection operation.  
Optimal choice of ADMM parameter and its validity for QPs with different scaling is presented 
in \S\ref{sec:optbetaFeasqp}.  \S~\ref{sec:nonstrictqp} verifies the claims of the paper on 
an example where strict complementarity is not satisfied at the solution. 
Finally, the behavior of ADMM iterates on infeasible 
QPs is presented in \S\ref{sec:Infeasqp}.

\subsection{Feasible QP - Tightness of worst-case convergence rates}\label{sec:tightFeasqp}
Consider the following QP,
\beq\bal
\min\limits_{\bfy =(y_1,y_2)} 
&\; \half \bfy^T\bfy  + \begin{bmatrix} 0 & -3 \end{bmatrix}\bfy \\
\text{s.t.} &\; \begin{bmatrix} 1 & 1 \end{bmatrix}\bfy = 1,\; \bfy \geq 0.
\eal\label{qpex1}\eeq
The optimal solution to the QP in~\eqref{qpex1} is $\bfy^* = (0, 1)$ and the multipliers for the 
non-negativity constraints are $\bflambda^* = (2,0)$.  
Thus, $\bfR = \frac{1}{\sqrt{2}}[1 \; 1]^T$ and $\bfE^* = [1 \;0]^T$.  
It is easily verified that the Assumptions~\ref{ass:feas}-\ref{ass:licq} are satisfied. 
Further, the cosine of the Friedrich's angle between $\bfR$ and $\bfE^*$ is 
$c_F^* = \frac{1}{\sqrt{2}}$.  Figure~\ref{fig:feasqp} plots the convergence rate $\|\bfv^k - \bfv^*\|/\|
\bfv^{k-1}-\bfv^*\|$ obtained from the ADMM iterations and the worst-case bound obtained in 
\S\ref{sec:onestepconvergence} against the iteration index $k$ for different choices of initial 
iterates.  The worst-case convergence bound of $\delta(\|\bfM_{\bfZ}\|,1.0,\alpha^{\max}(\|\bfv^k-\bfv^*\|))$ is plotted for iterations prior to $\activeset^k \subseteq \activeset^*$, while  
$\delta(\|\bfM_{\bfZ}\|,c_F^*,\alpha^k)$ is plotted for iterations following 
$\activeset^k \subseteq \activeset^*$. 
In all $3$ cases,  $\bfw^0 = (0,0)$, the ADMM step-size parameter $\beta$ is chosen as 
$1$ while $\bflambda^0$ is varied.  The ADMM iterations~\eqref{admmiter_ext} are 
executed until the 
satisfaction of the optimality termination conditions in~\eqref{opttermconds} to a tolerance of $
\epsilon_o = 10^{-6}$.  Figure~\ref{fig:lam1} shows that the convergence bound 
$\delta(\|\bfM_{\bfZ}\|,c_F^*,1.0)$ is attained at some ADMM iterations, thus the  
bound is indeed tight. The worst-case convergence rate bound $\delta(\|\bfM_{\bfZ}\|,1.0,\alpha^{\max}(\|\bfv^k-\bfv^*\|))$ is shown to be tight for iterations prior to the identification 
of the active-set in Figure~\ref{fig:lam100}.  Figure~\ref{fig:lam10} shows that for the choice of 
$\bflambda^0 = (30,30)$, the bound is not attained by the ADMM algorithm for any iteration.

\begin{figure}[ht]
\centering
\subfigure[$\bflambda^0 = (3,3)$]{
\includegraphics[scale=0.26]{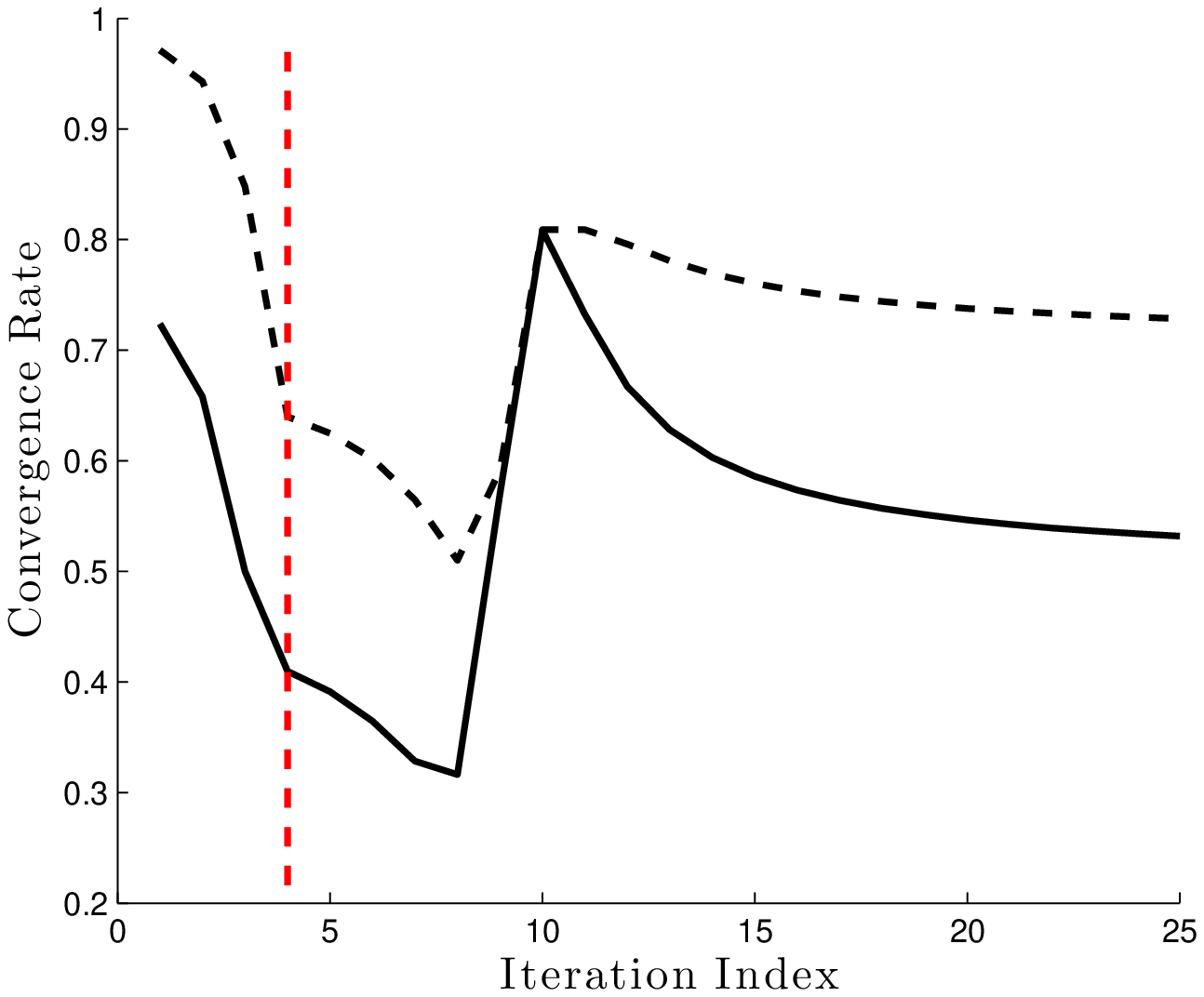}\label{fig:lam1}}
\subfigure[$\bflambda^0 = (30,30)$]{
\includegraphics[scale=0.26]{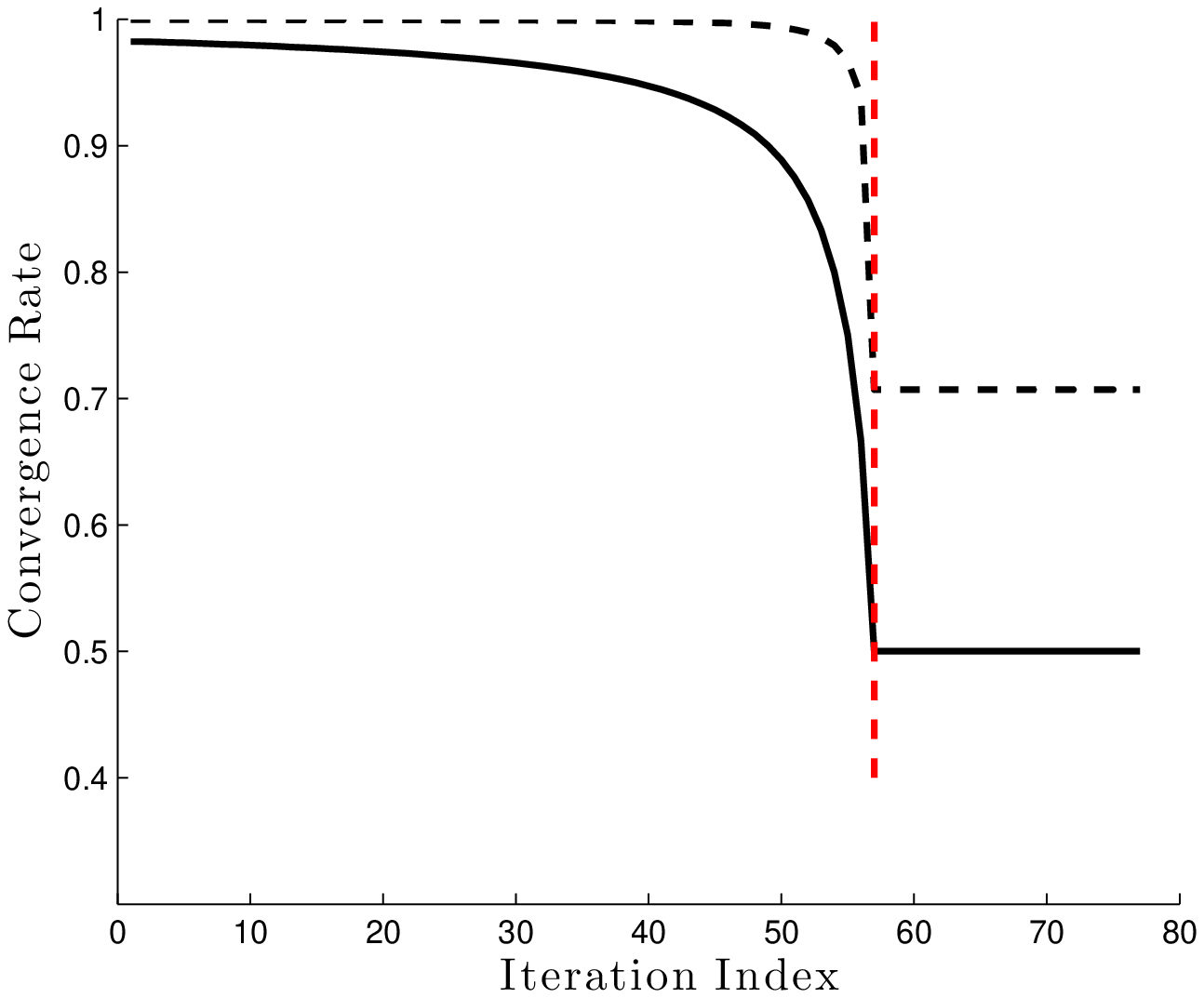}\label{fig:lam10}}
\subfigure[$\bflambda^0 = (300,300)$]{
\includegraphics[scale=0.26]{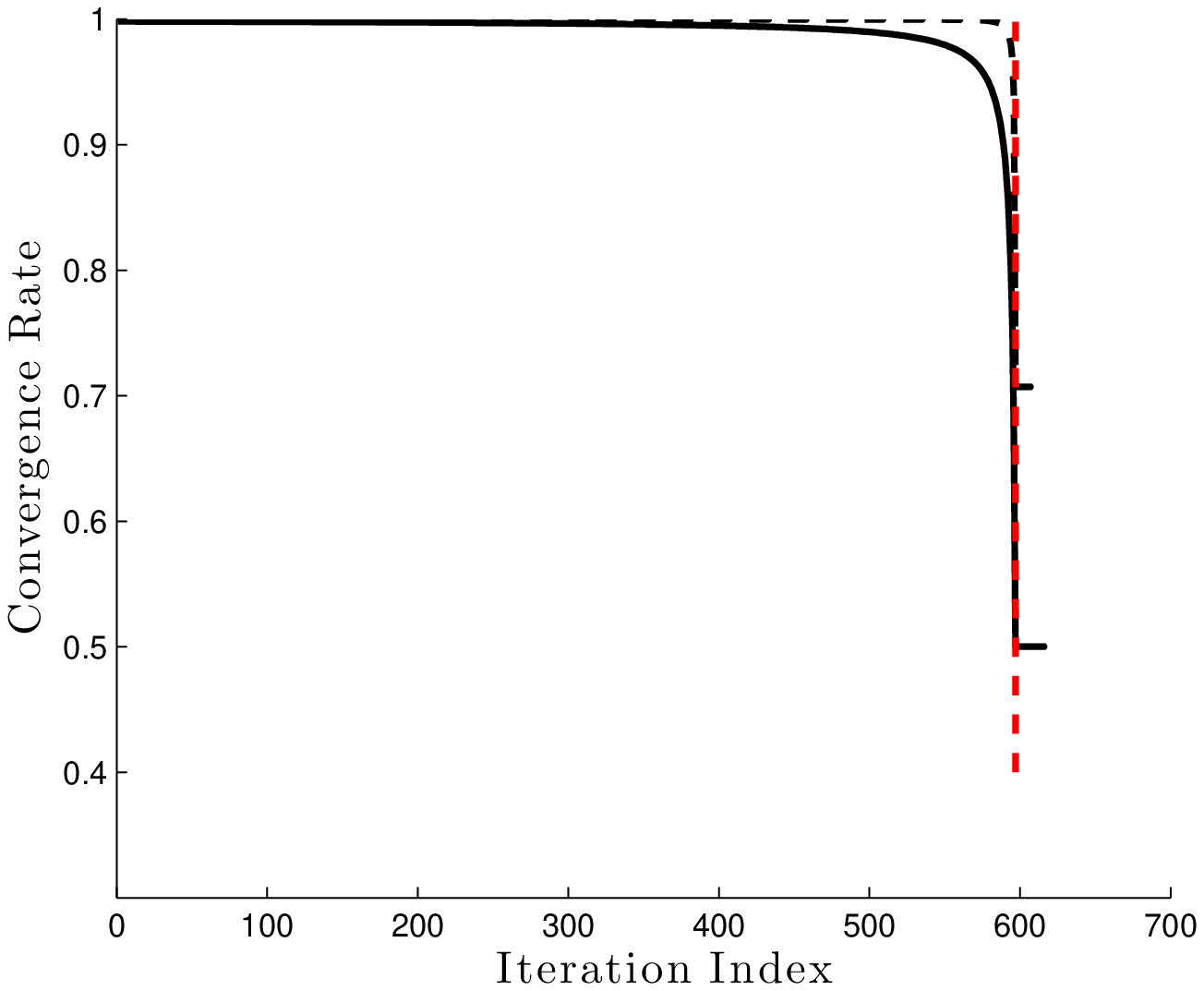}\label{fig:lam100}}
\label{fig:feasqp}
\caption{Plots of convergence rate vs. iteration index for different choices of initial iterate for the multipliers $\bflambda^0$. The solid curve is the actual ratio obtained from the ADMM iterations. The dashed curve is the worst-case bound on the convergence rate derived using the analysis in \S\ref{sec:onestepconvergence}. The vertical line indicates the iteration after which 
$\activeset^k \subseteq \activeset^*$.}
\end{figure}

\subsection{Feasible QP - Effect of problem scaling}\label{sec:scalingFeasqp}
Consider the following scaling of the QP in~\eqref{qpex1},
\beq\bal
\min\limits_{\bfy=(y_1,y_2)} 
&\; \half \bfy^T\begin{bmatrix} \kappa_1^2 & 0 \\ 0 & \kappa_2^2 \end{bmatrix}\bfy  + \begin{bmatrix} 0 & -3\kappa_2 \end{bmatrix}\bfy \\
\text{s.t.} &\; \begin{bmatrix} \kappa_1 & \kappa_2 \end{bmatrix}\bfy = 1,\; \bfy \in [0,\infty)
\eal\label{qpex2}\eeq
where $\kappa_1,\kappa_2 > 0$. 
It is easily seen that the optimal solution to the QP in~\eqref{qpex2} is 
$\bfy^* = (0,\frac{1}{\kappa_2})$, while the optimal multiplier for the non-negativity bound 
constraints is $\bflambda^* = ({2}{\kappa_1},0)$.  The choice of ADMM parameter $\beta$ and 
the convergence criterion are chosen as described in \S\ref{sec:tightFeasqp}. 
Figure~\ref{fig:scalingqp} plots the convergence rates (observed and worst-case bounds) for different values of the scaling parameters $\kappa_1,\kappa_2$.  The initial iterates are set as, 
$\bfw^0  = (0,0)$ and $\bflambda^0 = (3,3)$.  Figure~\ref{fig:scaleqp2}  
plots the convergence rates when $\kappa_1$ is increased while keeping $\kappa_2$ constant.  
As $\kappa_1$ is increased, $c_F^*$ increases from $0.707$ for $\kappa_1 = \kappa_2 = 1$ to 
$0.995$ for $\kappa_1 = 10$.  In other words, increasing $\kappa_1$ results in loss of LICQ 
(Assumption~\ref{ass:licq}) at the solution and this in turn increases $\delta(\|\bfM_{\bfZ}\|,c_F^*,
\alpha^k)$.  Figure~\ref{fig:scaleqp2} shows  that observed convergence rates are quite 
close to $1$ resulting in increased number of iterations for 
convergence once $\activeset^*$ has been identified.  On the other hand increasing $\kappa_2$ (refer Figure~\ref{fig:scaleqp3}) results 
in $c_F^*= 0.0995$. In others words, $\bfE^*$ is close to being in the null space of the constraints.  This results in smaller  $\delta(\|\bfM_{\bfZ}\|,c_F^*,\alpha^k)$ and faster convergence. Figure~\ref{fig:scaleqp3} shows that the worst-case bound 
estimate is tight once $\activeset^*$ is identified.  Increasing $\kappa_2$ further to $100$ has the effect  of further decreasing  
$\delta(\|\bfM_{\bfZ}\|,c_F^*,\alpha^k)$.  However, it has the undesirable consequence of reducing 
$\Delta\bfy^*_{i^{\min}}$ to $0.02$ as opposed to $2$ for $\kappa_2 = 1$. Thus,  
$\alpha^{\max}(\|\bfv^k-\bfv^*\|)$ is larger and the worst-case bound on convergence rate is 
now larger for iterations where $\activeset^k \nsubseteq \activeset^*$.  Consequently, more  
iterations are required to identify $\activeset^*$: $~30$ iterations for $\kappa_2 = 10$ 
(refer Figure~\ref{fig:scaleqp3}), $~300$ iterations for $\kappa_2 = 100$ 
(refer Figure~\ref{fig:scaleqp4}) as opposed to $~4$ iterations for $\kappa_2 = 1$ 
(refer Figure~\ref{fig:lam1}).  However, fewer iterations are required for convergence 
once $\activeset^*$ has been identified - $~20$ iterations for $\kappa_2 = 10$ (refer Figure~\ref{fig:scaleqp3}) and $~10$ iterations for $\kappa_2 = 100$ (refer Figure~\ref{fig:scaleqp4}).

\begin{figure}[ht]
\centering
\subfigure[$\kappa_1 = 10,\kappa_2 = 1$]{
\includegraphics[scale=0.26]{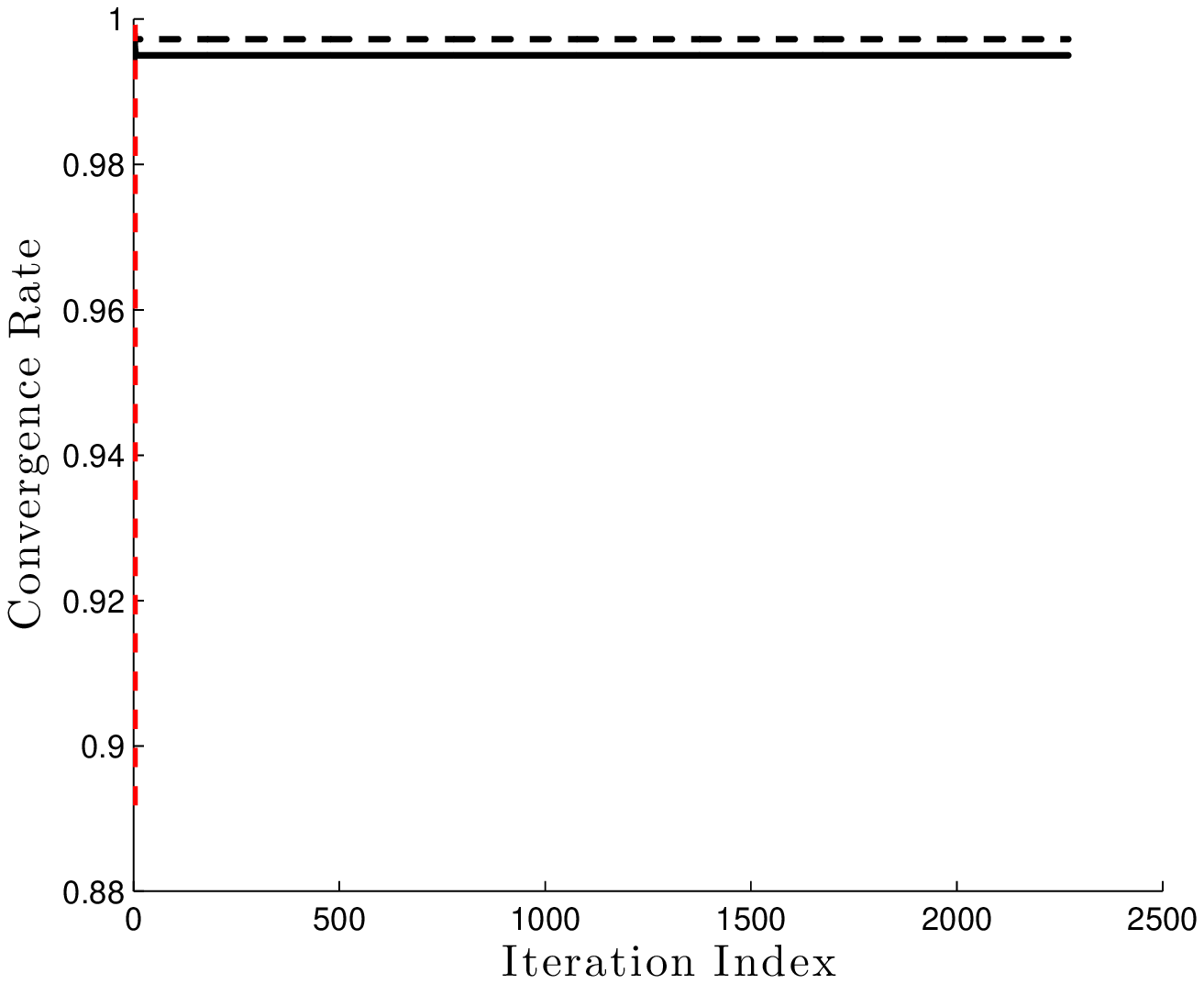}\label{fig:scaleqp2}}
\subfigure[$\kappa_1 = 1,\kappa_2 = 10$]{
\includegraphics[scale=0.26]{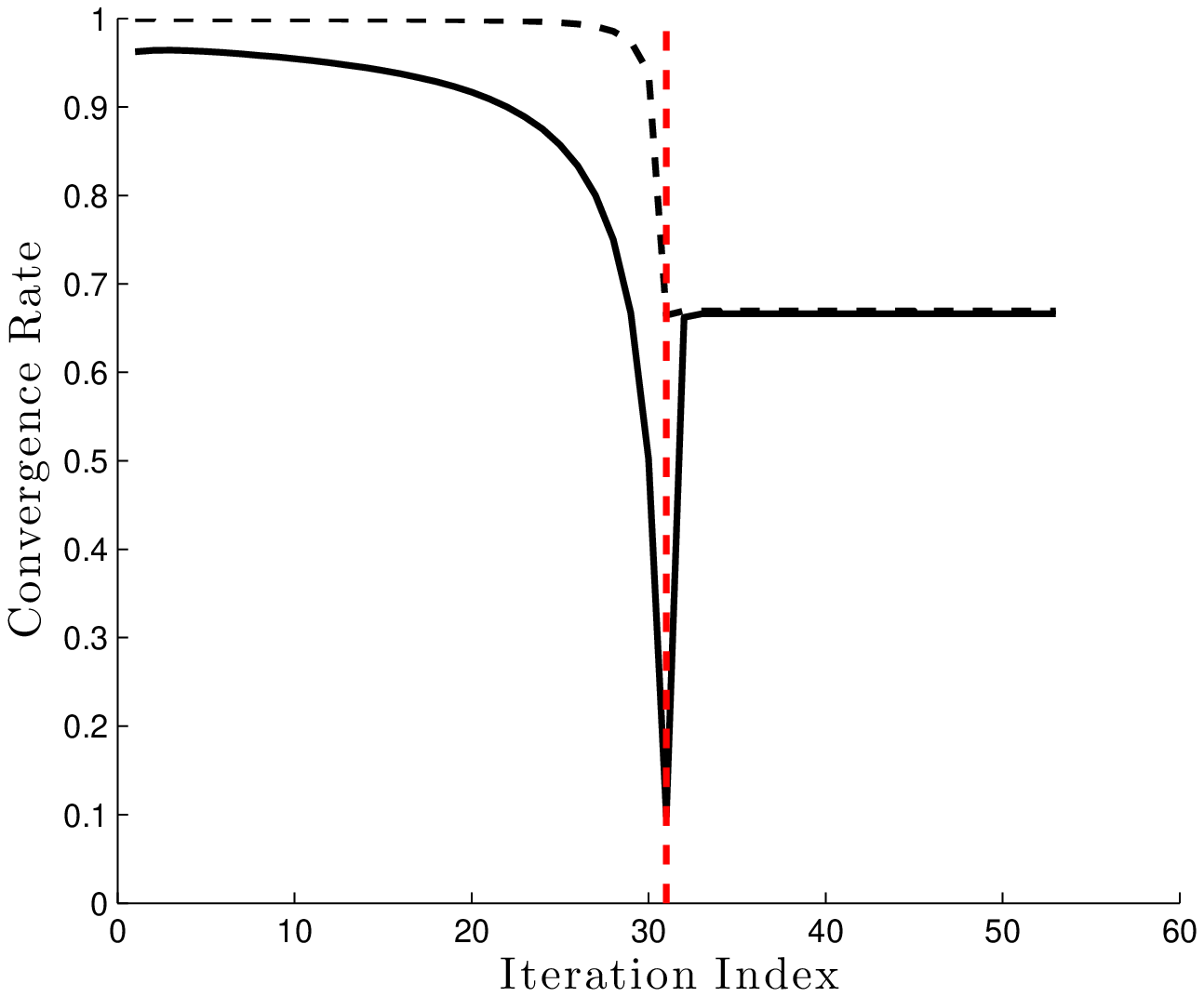}\label{fig:scaleqp3}}
\subfigure[$\kappa_1 = 1,\kappa_2 = 100$]{
\includegraphics[scale=0.26]{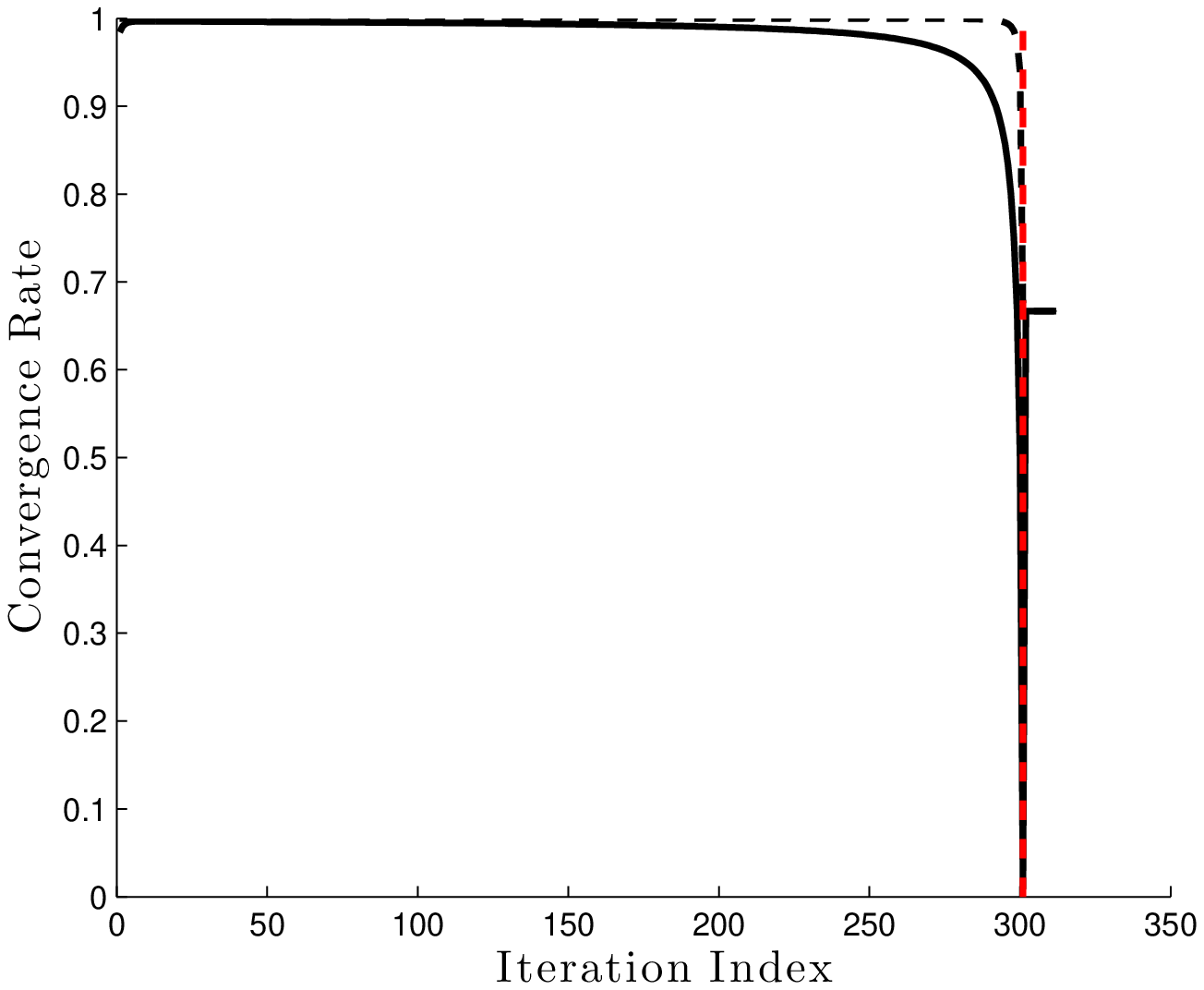}\label{fig:scaleqp4}}
\label{fig:scalingqp}
\caption{Plots of convergence rate vs. iteration index for different choices of scaling parameters $\kappa_1,\kappa_2$. The solid curve is the actual ratio obtained from the ADMM iterations. The dashed curve is the worst-case bound on the convergence rate derived using the analysis in \S\ref{sec:onestepconvergence}. The vertical line indicates the iteration after which 
$\activeset^k \subseteq \activeset^*$.}
\end{figure}

\subsection{Feasible QP - Optimal parameter choice $\beta^*$}\label{sec:optbetaFeasqp}
In this section, we consider the impact of ADMM parameter on the number of iterations for 
attaining convergence.  Figure~\ref{fig:optbetaqp} plots the iterations for convergence and 
iterations to identifying $\activeset^*$ against different values of ADMM parameter $\beta$.  
The choice of ADMM parameter $\beta$ and 
the convergence criterion are chosen as described in \S\ref{sec:tightFeasqp}. 
The initial iterates are set as, $\bfw^0  = (0,0)$ and $\bflambda^0 = (3,3)$ for all choices of the 
ADMM parameter. The optimal ADMM 
parameter for the cases depicted in Figure~\ref{fig:optbetaqp} are: (a) $\beta^* = 1$, 
(b) $\beta^* = 1.98$ and (c) $\beta^* = 1.98$.  The optimal choice coincides with the 
smallest number of iterations for both cases in Figures~\ref{fig:optbetaSc1} 
and~\ref{fig:optbetaSc3} for which the Assumptions of the paper are satisfied.  
However in the case of Figure~\ref{fig:optbetaSc2} where $c_F^*$ approaches $1$ 
which implies failure of LICQ, the proposed $\beta^*$ results in about an order of 
magnitude more iterations than the $\beta$ for which the 
ADMM algorithm converges in fewest number of iterations occur.  In general, it seems 
that more iterations are required for the identification of $\activeset^*$ as $\beta$ increases.  

\begin{figure}[ht]
\centering
\subfigure[$\kappa_1 = 1,\kappa_2 = 1$]{
\includegraphics[scale=0.26]{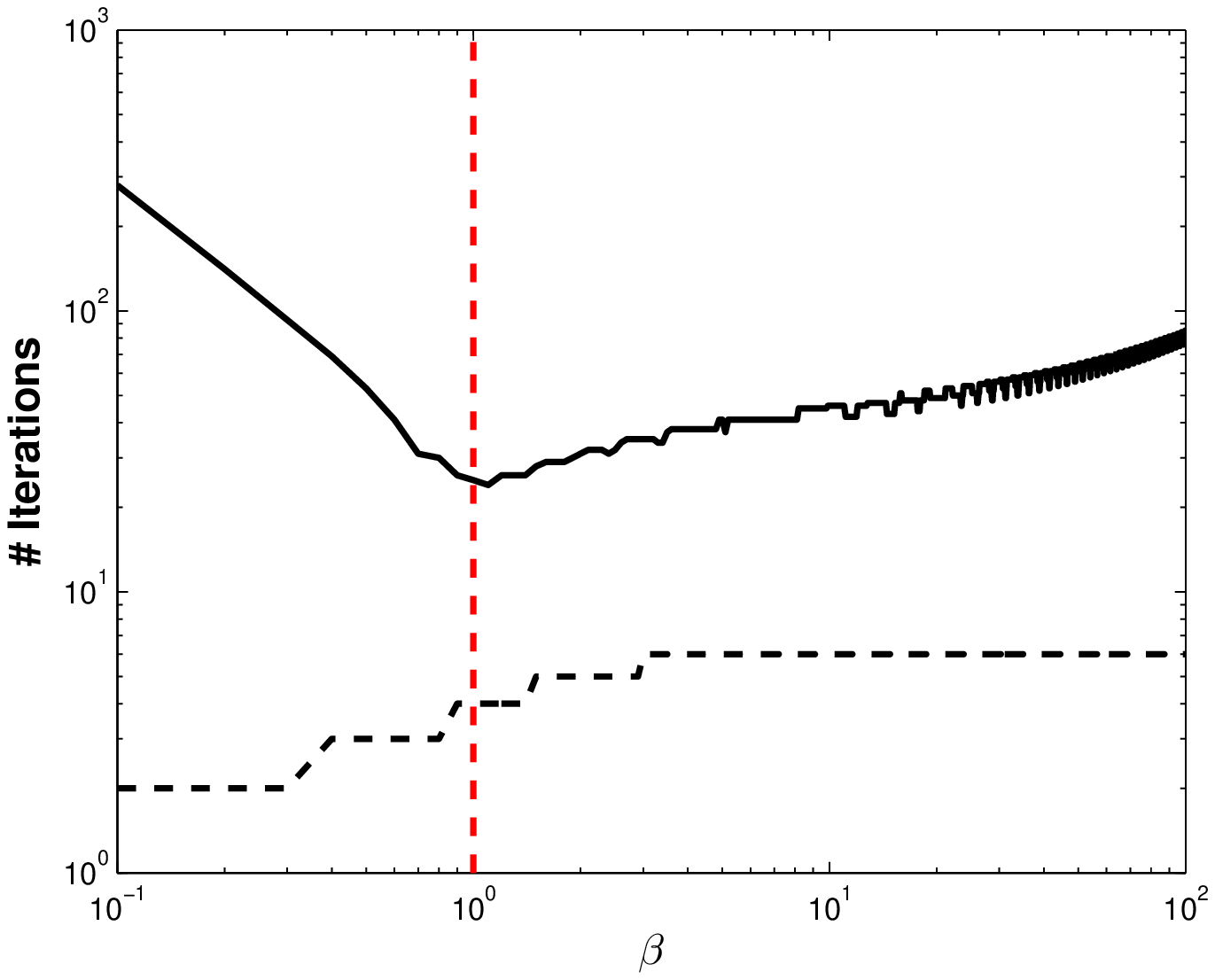}\label{fig:optbetaSc1}}
\subfigure[$\kappa_1 = 10,\kappa_2 = 1$]{
\includegraphics[scale=0.26]{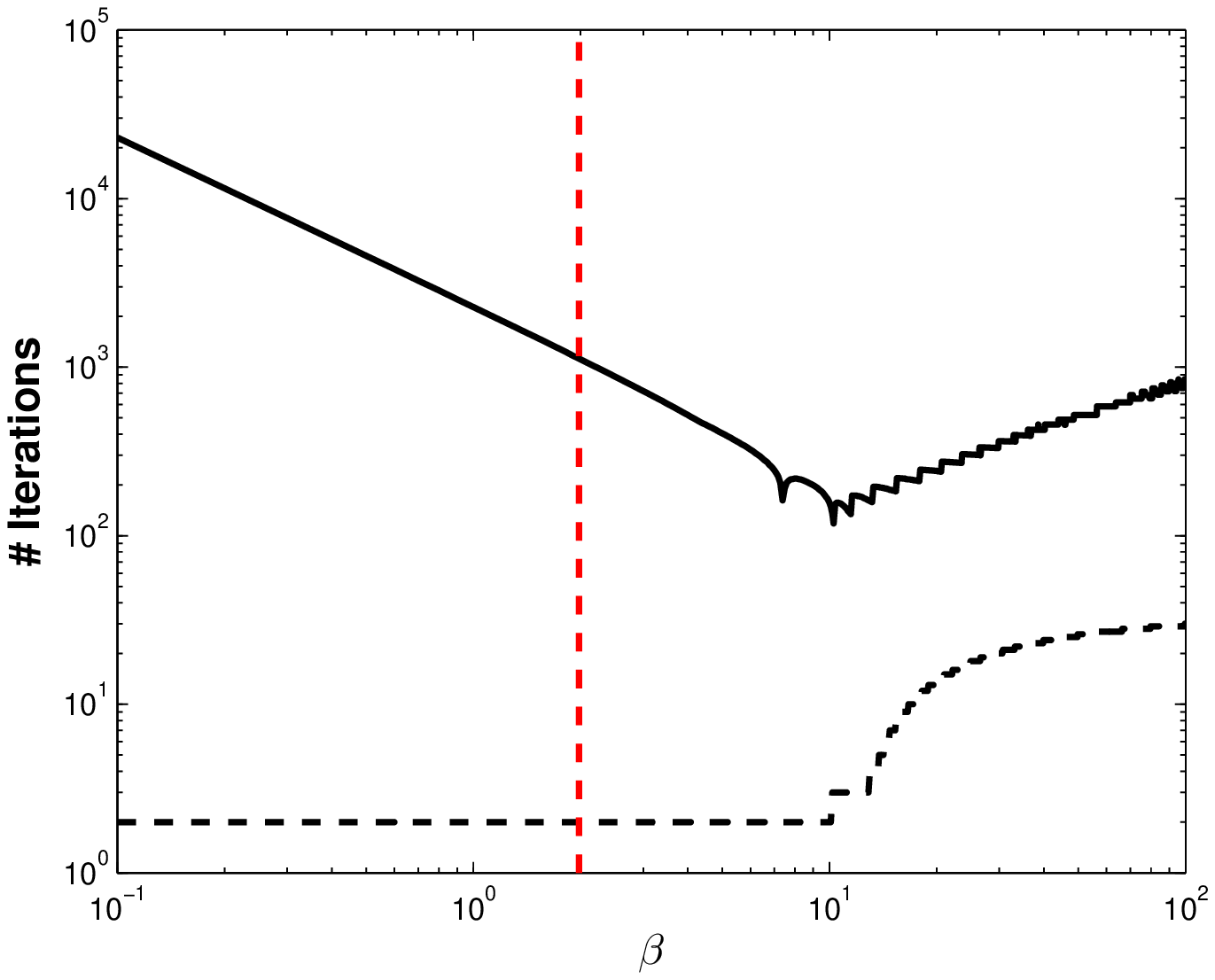}\label{fig:optbetaSc2}}
\subfigure[$\kappa_1 = 1,\kappa_2 = 10$]{
\includegraphics[scale=0.26]{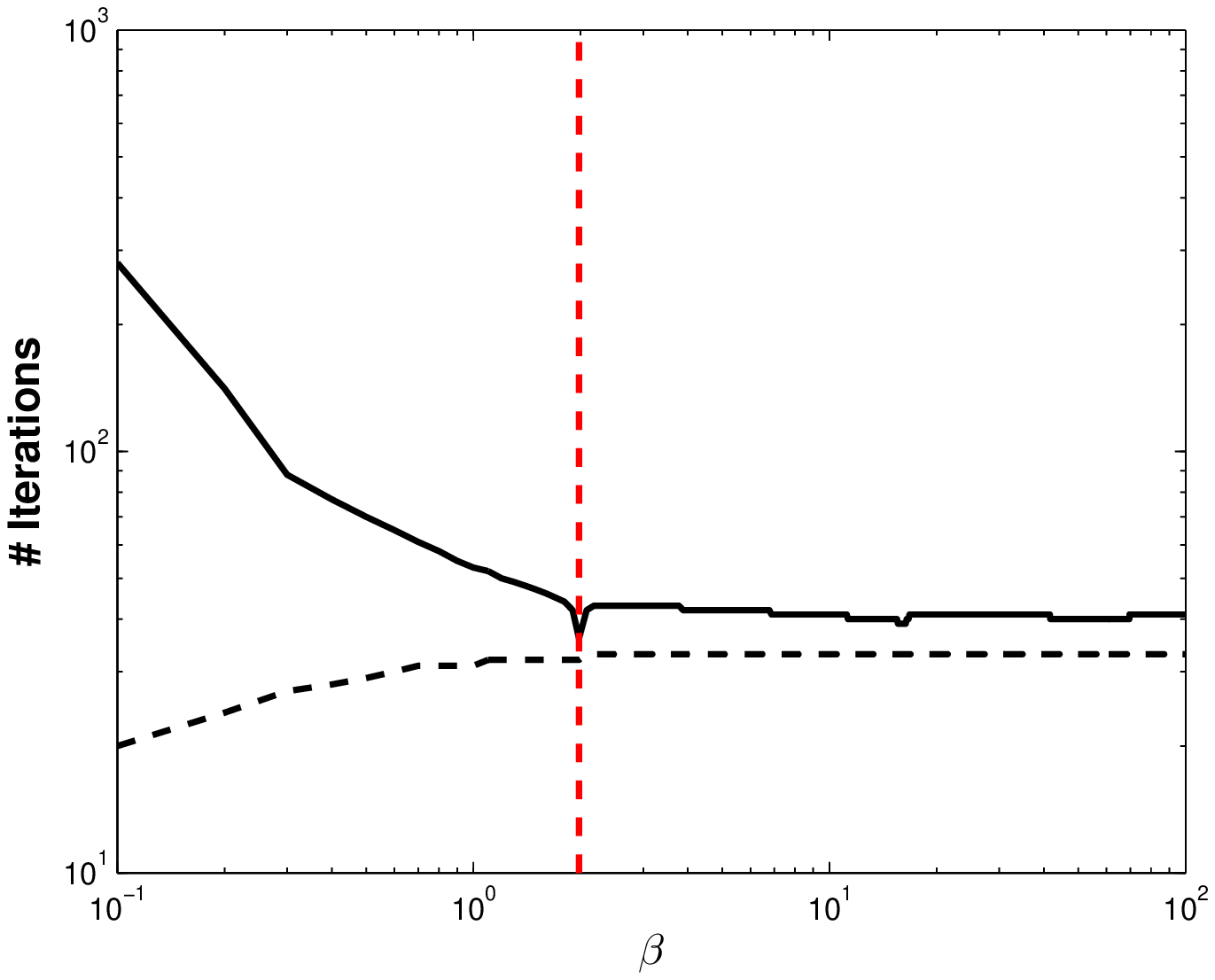}\label{fig:optbetaSc3}}
\label{fig:optbetaqp}
\caption{Plots of number of iterations to convergence vs. ADMM parameter value $\beta$ 
for different values of scaling parameters $\kappa_1,\kappa_2$. The solid curve is the iteration for 
convergence of the ADMM iterations. The dashed curve plots the iteration after which  
$\activeset^k \subseteq \activeset^*$.  The vertical line is the optimal ADMM parameter $\beta^* = \sqrt{\lambda_{\min}(\bfZ^T\bfQ\bfZ)\bflambda_{\max}(\bfZ^T\bfQ\bfZ)}$.}
\end{figure}

\subsection{Feasible QP - Non-strict complementarity}\label{sec:nonstrictqp}
In this section, we show that our results continue to hold on QP's where strict 
complementarity does not hold at the solution.  Consider modifying the 
QP~\eqref{qpex2} as $\bfq = [-2 \; -3]^T$.  The solution to this problem is 
$\bfy^* = (0,\frac{1}{\kappa_2})$ with multipliers $\bflambda^* = (0,0)$.  
Thus, the solution does not satisfy strict complementarity. 
The initial iterates are set as, $\bfw^0  = (0,0)$ and $\bflambda^0 = (3,3)$ for all choices of the 
ADMM parameter. The convergence criterion are chosen as described in 
\S\ref{sec:tightFeasqp}.   Figure~\ref{fig:noSc1} plots the convergence rate bound against 
observed convergence for different scaling parameters.  It is clearly seen that our analysis 
yields a strict contraction even when strict complementarity does not hold.  Further, 
Figure~\ref{fig:noSc2} also shows that estimate is tight for the case of non-strict 
complementarity.  Figure~\ref{fig:noSc3} shows that the characterization of optimal ADMM 
parameter also holds in the case of non-strict complementarity.

\begin{figure}[ht]
\centering
\subfigure[$\kappa_1 = 1,\kappa_2 = 1$]{
\includegraphics[scale=0.26]{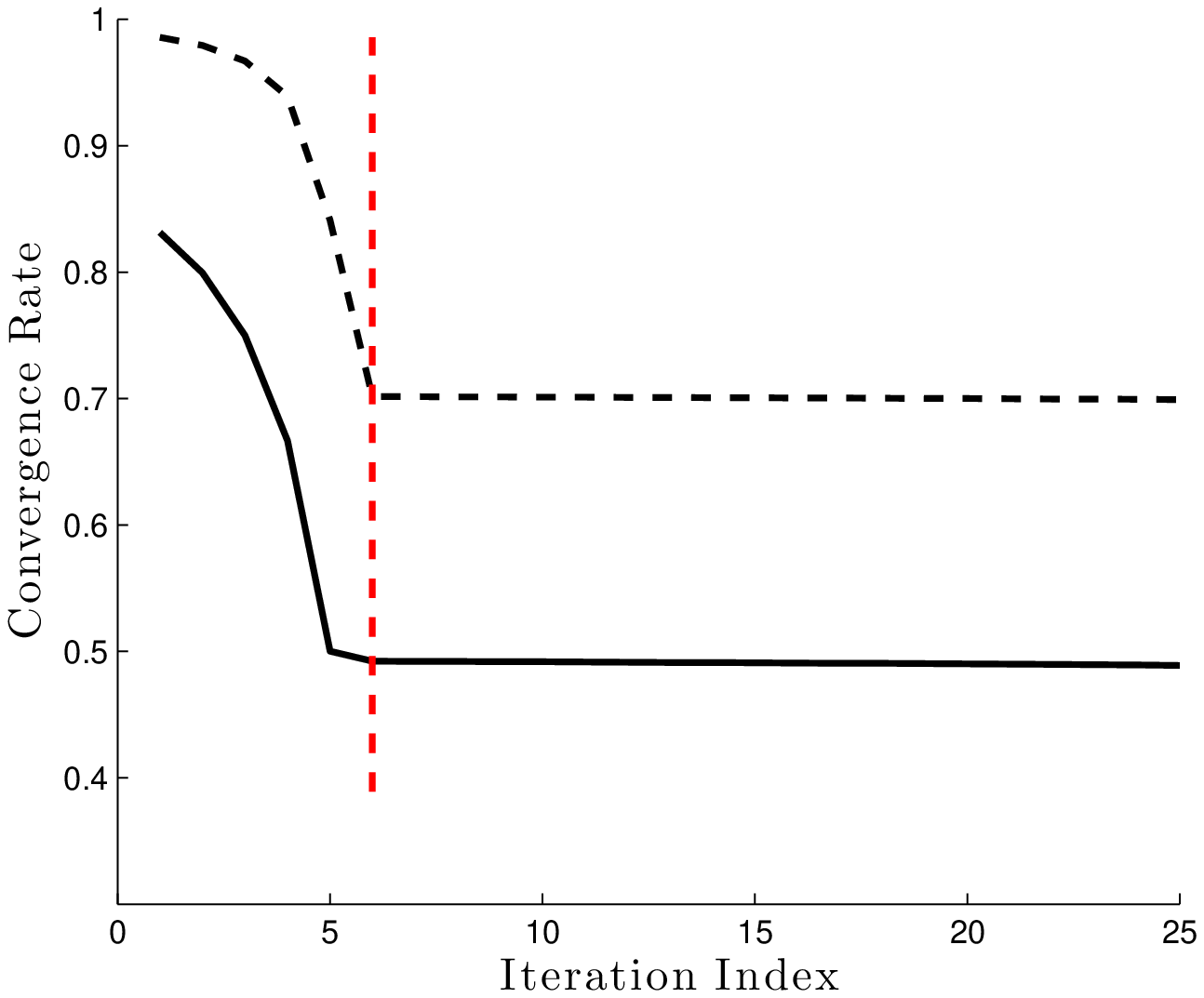}\label{fig:noSc1}}
\subfigure[$\kappa_1 = 10,\kappa_2 = 1$]{
\includegraphics[scale=0.26]{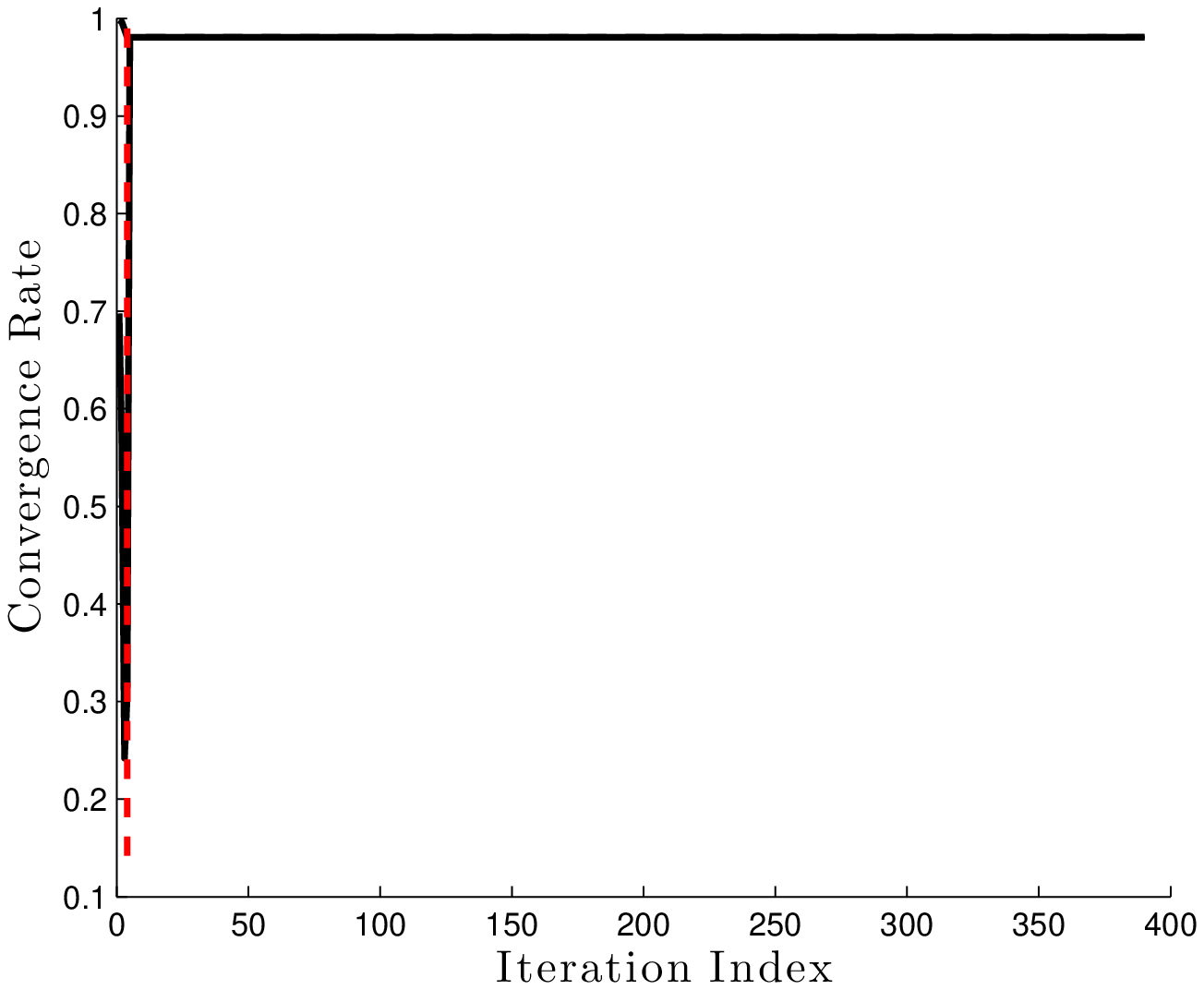}\label{fig:noSc2}}
\subfigure[$\kappa_1 = 1,\kappa_2 = 1$]{
\includegraphics[scale=0.26]{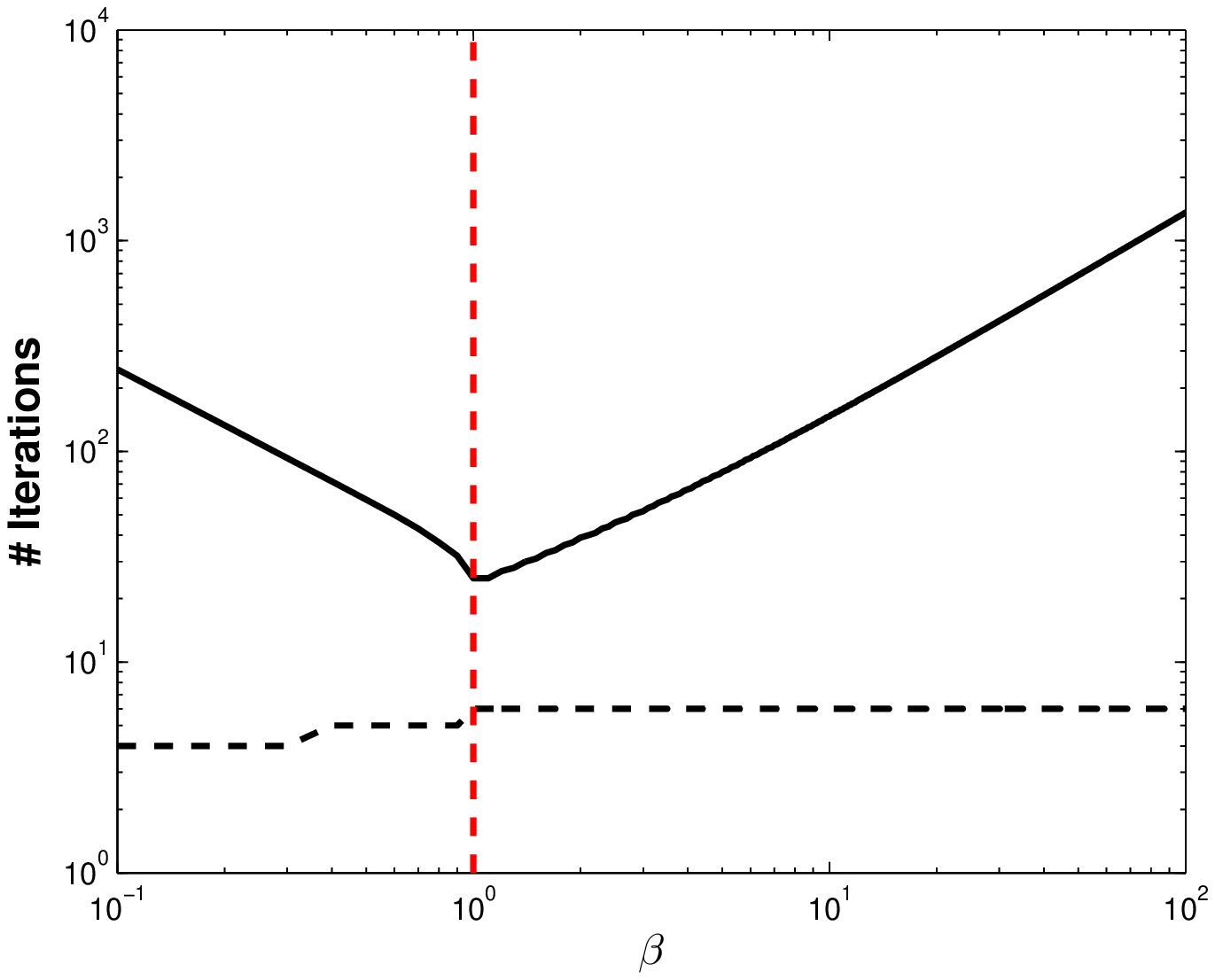}\label{fig:noSc3}}
\label{fig:qpnosc}
\caption{Figures~\ref{fig:noSc1} and~\ref{fig:noSc2} plots the convergence rate against 
iteration index. The solid curve is the actual ratio obtained from the ADMM iterations. The dashed curve is the worst-case bound on the convergence rate derived using the analysis in \S\ref{sec:onestepconvergence}. The vertical line indicates the iteration after which  
$\activeset^k \subseteq \activeset^*$.
Figure~\ref{fig:noSc3} plots the variation in iterations to reach convergence  against ADMM 
parameter $\beta$. The solid curve is the iteration for 
convergence of the ADMM iterations. The dashed curve plots the iterations after which 
$\activeset^k  \subseteq \activeset^*$.  The vertical line is the optimal ADMM parameter $\beta^* = \sqrt{\lambda_{\min}(\bfZ^T\bfQ\bfZ)\bflambda_{\max}(\bfZ^T\bfQ\bfZ)}$.}
\end{figure}

\subsection{Infeasible QP}\label{sec:Infeasqp}
Consider, the following infeasible QP
\beq\bal
\min\limits_{\bfy=(y_1,y_2)} 
&\; \half \bfy^T\bfy  + \begin{bmatrix} 0 & -3 \end{bmatrix}\bfy \\
\text{s.t.} &\; \begin{bmatrix} 1 & -1 \end{bmatrix}\bfy = -1,\; \bfy \in [-2,2]\times [5,10].
\eal\label{qpex3}\eeq
The QP above has an unique minimizer of infeasibility $\bfy^{\circ} = (3,4)$, $\bfw^{\circ} = (2,5)$ 
and $\bflambda^{\circ} = (-1,1)$ which implies that $\bfy^{\bfQ} = \bflambda^{\bfQ} = 0$. Figure~\ref{fig:betainfeasqp} plots various iteration related quantities over $100$ iterations of ADMM algorithm.  In both cases the initial iterates are chosen as, $\bfw^0 = \bflambda^0 = (0,0)$.  
of ADMM iteration Figure~\ref{fig:betainfeasqp} shows that $\{\bfy^k\} \rightarrow \bfy^\circ$ and $\{\bfw^k\} \rightarrow \bfw^{\circ}$ for both values of $\beta$ as shown in 
Theorem~\ref{thm:fixpointsinfeas}.  Further, convergence of 
$\{\bfv^k-\bfv^{k-1}\} \rightarrow -\bflambda^{\circ}$ is also verified as shown in Lemma~\ref{lemm:limitofvk}.  Finally, it also verified that the quantity~\eqref{infeasangleconds} holds in the fourth panel of Figures~\ref{fig:infeasqpBeta1} and~\ref{fig:infeasqpBeta10}.  Further, the larger value of $\beta$ results in faster convergence of different quantities towards the limiting values.  
It can also be verified that for the constraints in~\eqref{qpex3} the limiting values are independent 
of the choice of objective function. 

Consider modifying the equality constraint in~\eqref{qpex3} to $y_2 = 1$.   In this case, 
$\bfy^{\circ} = (a,1)$, $\bfw^{\circ} = (a,5)$ for any $a \in [-2,2]$ is a candidate for minimizer 
of Euclidean distance between the hyperplane and bound constraints.  Denote by 
$\bfy^{\circ} = (0,1)$ and $\bfw^{\circ} = (0,5)$.  For this definition of $\bfy^{\circ},\bfw^{\circ}$, 
as the linear term in the objective is varied - $[q_1 \; -3]^T$ it can be shown that
\[
\bfy^{\bfQ} = \left\{ \bal
(2,0) &\;\forall\; q_1 \leq -2 \\
(q_1,0) &\;\forall\; q_1 \in (-2,2) \\
(-2,0) &\;\forall\; q_1 \geq 2
\eal\right., \text{ and }
\bflambda^{\bfQ} = \left\{ \bal
(q_1+2,0) &\;\forall\; q_1 \leq -2 \\
(0,0) &\;\forall\; q_1 \in (-2,2) \\
(q_1-2,0) &\;\forall\; q_1 \geq 2.
\eal\right.
\]
The convergence of ADMM iterates to the limit defined in Theorem~\ref{thm:fixpointsinfeas} 
can be verified for different values of $q_1$.

\begin{figure}[ht]
\centering
\subfigure[$\beta = 1$]{
\includegraphics[width=0.48\textwidth,height=6cm]{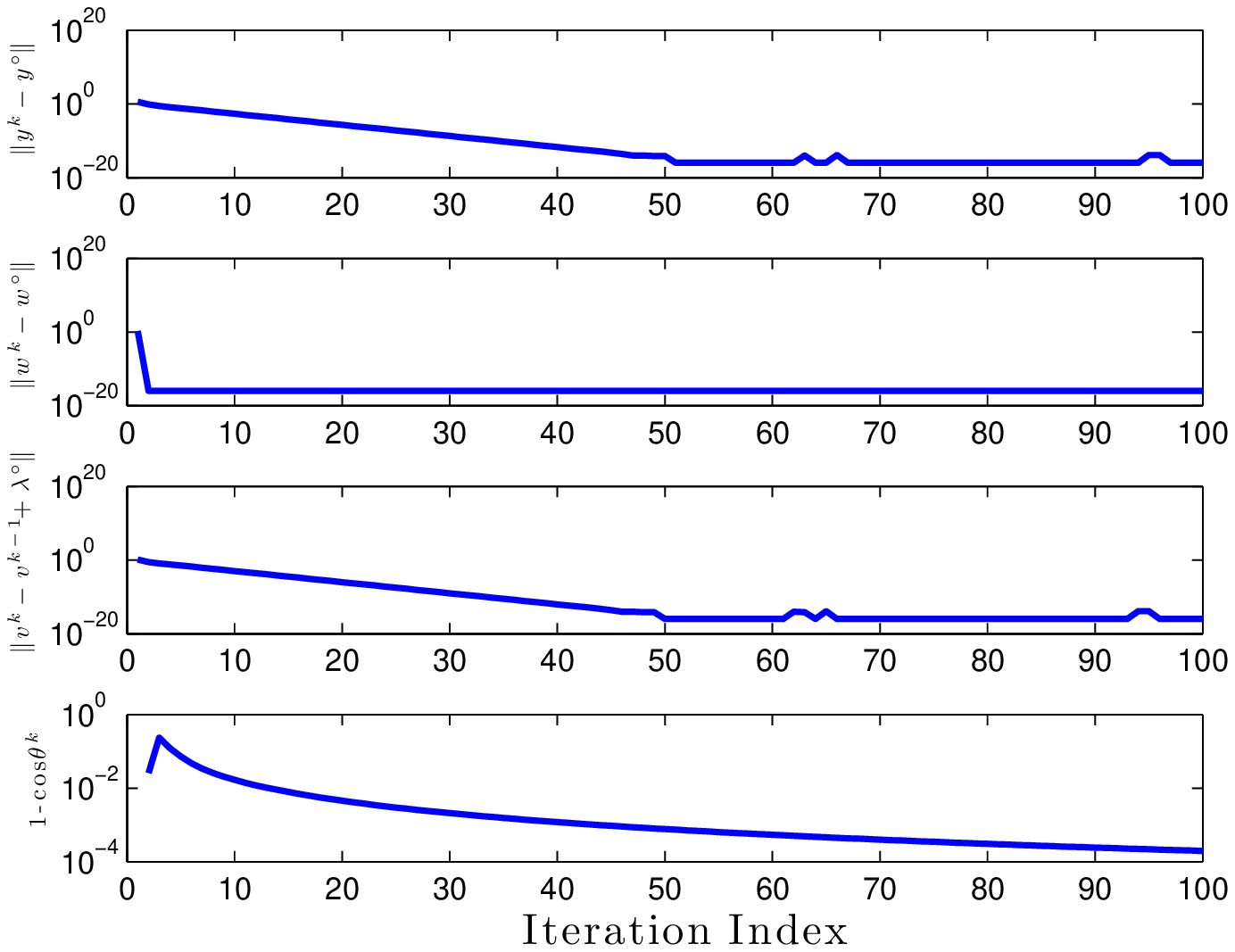}\label{fig:infeasqpBeta1}}
\subfigure[$\beta = 10$]{
\includegraphics[width=0.48\textwidth,height=6cm]{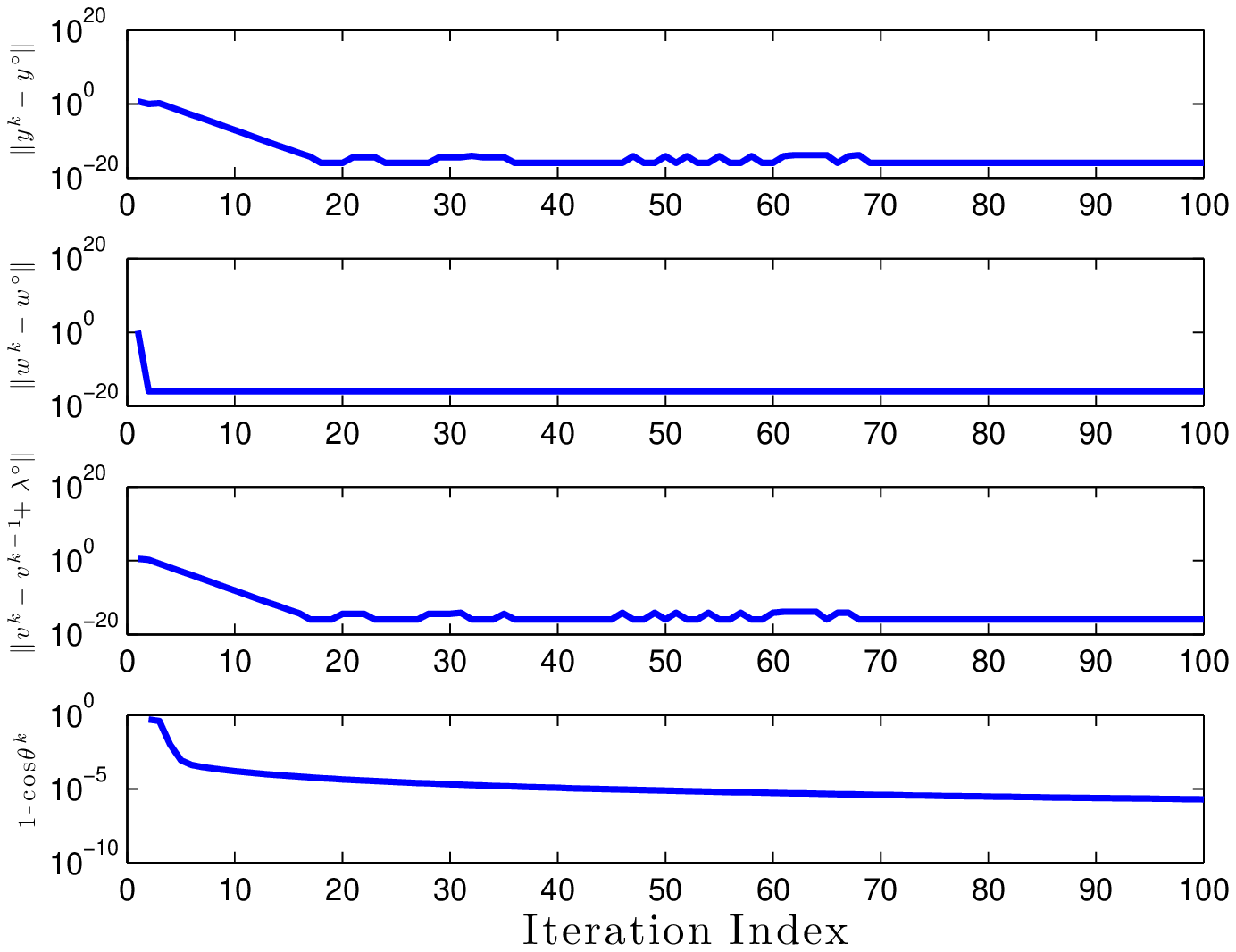}\label{fig:infeasqpBeta10}}
\label{fig:betainfeasqp}
\caption{Plots of different ADMM iteration related quantities vs. iteration index for different values of ADMM parameter $\beta$. In the above, cos($\theta^k$) = $(\bflambda^k)^T(\bfw^k-\bfy^k)/(\|\bflambda^k\|\cdot\|\bfw^k-\bfy^k\|)$.}
\end{figure}

\section{Conclusions}\label{sec:conclusions}
The paper analyzes convergence behavior of the ADMM iterations for convex QPs.  The algorithm is shown to be Q-linearly convergent under positive definiteness of reduced Hessian and linear independence constraint qualifications for feasible QPs.  For infeasible QPs, we analyze the limit of the sequence of iterates generated by ADMM 
and provide conditions for detecting infeasibility.

\section*{Acknowledgement}
The authors like to thank Pontus Giselsson for pointing out the  incomplete proofs of a conference submission which led to the present approach to showing the current results, and  also Andrew Knyazev for discussions on angle between subspaces.  The authors are also thankful to  the 
anonymous referees whose careful reading and detailed comments helped to improve the presentation.

\bibliographystyle{plain}

\end{document}